\documentclass[11pt]{amsart}
\usepackage{amsfonts,amssymb,amsmath,amsthm,graphicx,color,amscd,xspace,verbatim}
\usepackage[margin=1.1in]{geometry}

\newtheorem{theorem}{Theorem}[section]
\newtheorem{proposition}[theorem]{Proposition}

\newtheorem{lemma}[theorem]{Lemma}
\theoremstyle{definition}
\newtheorem{remark}[theorem]{Remark}
\newtheorem{definition}[theorem]{Definition}
\numberwithin{equation}{section}

 \newcommand{\<}{\left\langle}
\renewcommand{\>}{\right\rangle}

\newcommand{\eps}{\varepsilon}

\newcommand{\norm}[1]{\left\Vert#1\right\Vert}

\newcommand{\be} {\begin{equation}}
\newcommand{\ee} {\end{equation}}
\newcommand{\bea} {\begin{eqnarray}}
\newcommand{\eea} {\end{eqnarray}}
\newcommand{\Bea} {\begin{eqnarray*}}
\newcommand{\Eea} {\end{eqnarray*}}
\newcommand{\pa} {\partial}

\newcommand{\al} {\alpha}

\newcommand{\de} {\delta}

\newcommand{\ga} {\gamma}
\newcommand{\Ga} {\Gamma}
\newcommand{\Om} {\Omega}

\newcommand{\De} {\Delta}
\newcommand{\la} {\lambda}
\newcommand{\si} {\sigma}

\newcommand{\no} {\nonumber}
\newcommand{\noi} {\noindent}
\newcommand{\lab} {\label}
\newcommand{\va} {\varphi}
\newcommand{\f}{\frac}
\newcommand{\R}{\mathbb R}
\newcommand{\N}{\mathbb N}
\newcommand{\Rn}{\mathbb R^N}
\newcommand{\Iom}{\int_{\Omega}}
\newcommand{\tl}{\tilde}
\catcode`\@=11
\def\ga{\alpha}            \def\gg{\gamma}
       \def\gd{\delta}      
\def\gth{\theta}                         \def\vge{\varepsilon}
\def\gf{\phi}       \def\vgf{\varphi}    
      \def\gk{\kappa}      \def\gl{\lambda}

     \def\Gd{\Delta}

\def\Gw{\Omega}              

\def\CT{{\mathcal T}}

\def\BBG {\mathbb G}       
       
   \def\BBN {\mathbb N}

   \def\BBX {\mathbb X}

\def\GTM {\mathfrak M}

\newcommand{\supp}{\mathrm{supp}}
\def\fw{{(-\Gd)^s}}

\newcommand{\unl}{\underline}
\begin{document}

\title[Existence and multiplicity of solutions to fractional Lane-Emden  systems]{On the existence and multiplicity of solutions to fractional Lane-Emden elliptic systems involving measures}

\author{ Mousomi Bhakta}
\address{M. Bhakta, Department of Mathematics, Indian Institute of Science Education and Research, Dr. Homi Bhabha Road,
Pune 411008, India}
\email{mousomi@iiserpune.ac.in}

\author{ Phuoc-Tai Nguyen}
\address{P. T. Nguyen, Department of Mathematics and Statistics, Masaryk University, Brno, Czech Republic}
\email{ptnguyen@math.muni.cz}

\date{}
\begin{abstract} We study positive solutions to the fractional Lane-Emden system 
\begin{equation*} \tag{S}\label{S}  \left\{  \begin{aligned} 
(-\Delta)^s u &= v^p+\mu \quad &&\text{in } \Omega \\
(-\Delta)^s v &= u^q+\nu \quad &&\text{in } \Omega\\
u = v &= 0 \quad &&\text{in } \Omega^c={\mathbb R}^N \setminus \Omega, 
\end{aligned} 
\right. \end{equation*} 
where $\Omega$ is a $C^2$ bounded domains in $\Rn$, $s\in(0,1)$,  $N>2s$, $p>0$, $q>0$ and $\mu,\, \nu$ are positive measures in $\Om$. We prove the existence of the minimal positive solution of \eqref{S} under a smallness condition on the total mass of $\mu$ and $\nu$.  Furthermore, if $p,q \in (1,\frac{N+s}{N-s})$ and $0 \leq \mu,\, \nu\in L^r(\Om)$ for some $r>\frac{N}{2s}$ then we show the existence of at least two positive solutions of \eqref{S}. We also discuss the regularity of the solutions.
\end{abstract}

\maketitle

\tableofcontents

\section{Introduction and main results}
In this article we consider elliptic system of the type
\begin{equation} \label{eq:system} \left\{ \begin{aligned} 
(-\Delta)^s u &= v^p + \mu \quad &&\text{in } \Omega \\
(-\Delta)^s v &= u^q + \nu  \quad &&\text{in } \Omega \\
u,\,v &\geq 0 \quad &&\text{in } \Omega \\
u = v &= 0 \quad &&\text{in } \Omega^c=\R^N \setminus \Omega,
\end{aligned} 
\right. \end{equation}
where $\Gw$ is a $C^2$ bounded domain in $\Rn$, $s\in(0,1)$, $N>2s$, $p>0$, $q>0$ and  $\mu,\, \nu$ are positive Radon measures in $\Om$. 
%satisfying 
%$$\Iom \de(x)^s d\mu<\infty, \quad \Iom \de(x)^s d\nu<\infty,  \quad %\de(x)=\text{dist}(x, \pa\Om).$$ 
Here $(-\De)^s$ denotes the  fractional Laplace operator defined as follows
$$ (-\De)^s u(x)=\lim_{\eps\to 0}(-\De)^s_{\eps}u(x), $$
where
\begin{equation} \label{De-u}
  \left(-\Delta\right)_\vge^su(x): = a_{N,s}\int_{\mathbb{R}^N\setminus B_{\eps}(x)}\frac{u(x)-u(y)}{|x-y|^{N+2s}}dy,
\end{equation}
and  $a_{N,s}= \frac{2^{2s}s\Ga(N/2+ s)}{\pi^{N/2}\Ga(1-s)}$. 

When $s=1$, $\fw$ coincides the classical laplacian $-\Gd$ and the Lane-Emden system
\begin{equation} \label{loc-system} \left\{ \begin{aligned} 
-\Delta u &= v^p + \mu \quad &&\text{in } \Omega, \\
-\Delta v &= u^q + \nu\quad &&\text{in } \Omega, 
\end{aligned} 
\right. \end{equation}
has been studied extensively in the literature (see \cite{BVY, Co, FF, GN, HMMY, Mit-1, PQS,  RZ, SZ} and the references therein).
Bidaut-V\'eron and Yarur \cite{BVY} provided various necessary and sufficient conditions in terms of estimates on the Green kernel for the existence of solutions of \eqref{loc-system}.  When $\mu=\nu=0$, the structure of solution of \eqref{loc-system} has been better understood according to the relation between $p,q$ and $N$. More precisely, if  $\frac{1}{p+1}+\frac{1}{q+1}\leq \frac{N-2}{N}$ then \eqref{loc-system} admits some positive (radial, bounded) classical solutions in $\Rn$ (see \cite{SZ}). On the other hand, the so-called Lane-Emden conjecture states that if \be\lab{7-8-1}
\frac{1}{p+1}+\frac{1}{q+1}> \frac{N-2}{N}\ee
then there is no nontrivial classical solution of \eqref{loc-system} in $\Rn$. The conjecture is known to be true for radial solutions in all dimensions (see \cite{Mit-1}). In the nonradial case, partial results have been achieved. Nonexistence was proved in \cite{FF,RZ} for
$(p, q)$ in certain subregions of \eqref{7-8-1}.
%for example, when $$p,q\leq\frac{N+2}{N-2}, \quad %(p,q)\not=\bigg(\frac{N+2}{N-2}, \frac{N+2}{N-2}\bigg)$$ or when 
%$$P,\,Q\geq \frac{N-2}{2}, \quad P+Q>N-2,$$ where %$P=\frac{2(p+1)}{pq-1}$ and $Q=\frac{2(q+1)}{pq-1}$ (see \cite{FF, %RZ}). 

For nonlocal case, i.e. $s\in(0,1)$, Quaas and Xia \cite{QX3} showed the existence of at least one positive viscosity solution for the system of the type 
\begin{equation} \label{sys0} \left\{ \begin{aligned} 
(-\Delta)^s u &= v^p \quad &&\text{in } \Omega, \\
(-\Delta)^s v &= u^q \quad &&\text{in } \Omega, \\
u = v &= 0 \quad &&\text{in } \Omega^c.
\end{aligned} 
\right. \end{equation}
It has been proved in \cite{QX1} that under some conditions on the exponents $p$ and $q$, system \eqref{sys0} does not admit any positive bounded viscosity solution. 
%if $$\frac{N}{N-2s}<p,q\leq \frac{N+2s}{N-2s}, \quad (p,q)\not=\bigg(\frac{N+2s}{N-2s}, \frac{N+2s}{N-2s}\bigg)$$ or
%$$p,q>1,\quad \max\bigg\{\frac{2s(p+1)}{pq-1}, \frac{2s(q+1)}{pq-1}\bigg\}>N-2s.$$ 
We also refer \cite{DKK, QX2} for further results in this directions.

Nonlocal equations with measure data have been investigated in \cite{BN, CFV, CQ, CV, TV}  and the references therein. More precisely, fractional elliptic equations with interior measure data were studied  in \cite{CV,CQ}, while the equations with measure boundary data were carried out in \cite{TV} (for absorption nonlinearity)  and in \cite{BN} (for source nonlinearity).  \smallskip

To the best of our knowledge, there has been no result concerning nonlinear fractional elliptic systems with measure data in the literature so far. The present paper can be regarded as one of the first publications in this direction and our main contribution is the existence and multiplicity result (see Theorem \ref{2nd sol}) which is new even in the local case $s=1$. Our approach is based on a combination of the theory of PDEs with measure data and variational method (in particular Linking theorem). \smallskip

%For $\gb>0$, we set
%$$ \begin{aligned} \Gs_{\gb}:=\{x\in\Om: \de(x)=\ba\}, \quad \Gw_{\gb}:=\{x\in\Om: \de(x)<\ba\}, \quad D_{\ba}:=\{x\in\Om: \de(x)>\ba\}.
%\end{aligned} $$

Before stating the main results, we introduce necessary notations.

For $\phi\geq 0$, denote by $\mathfrak{M}(\Om,\phi)$ the space of Radon measures $\tau$ on $\Om$ satisfying $\int_{\Om}\phi\, d |\tau|<\infty$ and  by $\mathfrak{M}^+(\Om,\phi)$ the positive cone of $\mathfrak{M}(\Om,\phi)$. For $\kappa>0$, denote by $L^\kappa(\Omega,\phi)$ the space of measurable functions $w$ such that $\int_{\Omega}|w|^\kappa \phi dx <\infty$. We denote  $\de(x)=\text{dist}(x,\pa\Om)$. When $\phi=\delta^s$, we can define the space $\mathfrak{M}(\Om,\delta^s)$ and $L^\kappa(\Omega, \de^s)$. Let $G_s=G_s^\Omega$ be the Green kernel of $(-\De)^s$ in $\Om$. We denote the associated Green operator $\mathbb{G}_s$ as follows:
$$ \begin{aligned} \mathbb{G}_s[\tau](x) &:= \int_{\Om}G_s(x, y)d\tau(y), \quad \tau\in\mathfrak{M}(\Om,\de^s).
\end{aligned} $$
Important estimates concerning the Green kernel are presented  Section 2.

\begin{definition} \label{defsol}(Weak solution) Let $\mu,\, \nu \in \GTM^+(\Gw,\gd^s)$. We say that  $(u, v)$ is a weak solution of \eqref{eq:system} if $u,\, v \in L^1(\Gw)$, $v^p,\, u^q \in L^1(\Gw,\gd^s)$ and 
\begin{equation} \label{intN}\left\{ \begin{aligned}
 \int_{\Gw} u \fw \xi dx &= \int_{\Gw}v^p\xi dx + \int_{\Gw}  \xi \,d\mu, \\ %\quad \forall\, \xi \in \BBX_s(\Gw), \\
	\int_{\Gw} v \fw \xi dx &= \int_{\Gw}u^q\xi dx + \int_{\Gw}  \xi \,d\nu, %\quad \forall\, \xi \in \BBX_s(\Gw),
	\end{aligned}\right. \quad \forall \xi \in \BBX_s(\Gw),
	\end{equation}
\end{definition}
where $\BBX_s(\Gw)\subset C(\Rn)$ denotes the space of test functions $\xi$ satisfying

(i) $\supp (\xi) \subset \bar \Gw$,

(ii) $\fw \xi(x)$ exists for all $x \in \Gw$ and $|\fw \xi(x)| \leq C$ for some $C>0$,

(iii) there exists $\vgf \in L^1(\Gw,\gd^s)$ and $\varepsilon_0>0$ such that $|(-\Gd)_\varepsilon^s \xi| \leq \vgf$ a.e. in $\Gw$, for all $\varepsilon \in (0,\varepsilon_0]$.

\begin{remark} We observe that, by \cite[Proposition A]{TV}, $(u, v)$ is a weak solution of \eqref{eq:system} if and only if 
\begin{equation} \label{uGM}  u = \BBG_s[v^p] + \BBG_s[\mu]\quad\text{and}\quad v=\BBG_s[u^q] + \BBG_s[\nu].
\end{equation}
\end{remark}

%Our next result depicts the relation between weak solutions and viscosity solutions.
%
%\begin{theorem} \label{regularity4} Let $\mu, \, \nu \in \GTM^+(\prt \Gw)$ and $p,\, q \in (1,N_s)$, where $N_s$ be as in \eqref{p-s}. Assume $f, g \in C(\BBR^+)$ and there are positive constants $a,b$ such that 
%	\begin{equation} \label{condf}\begin{aligned} 0 \leq f(t) &\leq at^p + b,  \quad \forall t \geq 0\\
%	0 \leq g(t) &\leq a t^q + b,  \quad \forall t \geq 0.
%	\end{aligned}
%	\end{equation}
%	If $(u, v)$ is a nonnegative weak solution of \eqref{source} then $u,\, v \in C_{loc}^{2s+\alpha}(\Gw)$, for some $\alpha \in (0,1)$ and $(u, v)$ is a viscosity solution. 
%	
%	Moreover if system \eqref{8-8-1} does not admit any bounded, positive viscosity super-solution then $(u, v)$
%	satisfies \eqref{estsys2}. 
%\end{theorem}

Define \be\lab{p-s}N_s:=\frac{N+s}{N-s}. \ee

%\begin{theorem} \label{existf}  Let $\mu,\, \nu \in \GTM^+(\Gw,\gd^s)$ and $p,\, q \in (1,N_s)$, where $N_s$ be as in \eqref{p-s}. Assume $f,\, g \in C(\BBR^+)$ satisfies \eqref{condf}. 
%	There exist $\hat b$ and $\hat \rho$  such that if $b \in (0,\hat b)$, and $\max\{\| \mu \|_{\GTM(\prt \Gw)}, \| \nu \|_{\GTM(\prt \Gw)} \}<\hat \rho$, then problem \eqref{source} admits a nonnegative weak solution $(u,v)$ such that $u \geq \BBM_s[\mu]$ and $v \geq \BBM_s[\nu]$. 
%	
%	Moreover, if system \eqref{8-8-1} does not admit any bounded, positive viscosity super-solution then $(u, v)$ satisfies \eqref{estsys2}. 
%\end{theorem} 

Our first result is the existence of the minimal weak solutions of \eqref{eq:system}.
\begin{theorem} (Minimal solution) \label{existcoup}
	Let $p,q>0$ with $p\leq q$, $pq \neq 1$ and $q\frac{p +1}{q+1} <N_s$. 
	Assume $\mu, \nu \in \GTM^+(\Gw,\gd^s)$ and $\BBG_s[\mu]\in L^{q}(\Gw,\gd^{s})$. Then  system \eqref{eq:system} admits a positive weak solution $(\underline u_{\mu}, \underline v_{\nu})$ for $\|\mu\|_{\GTM(\Gw,\gd^s)}$ and $\|\nu\|_{\GTM(\Gw,\gd^s)}$ small if $pq>1$ and  for any $\mu,\nu \in \GTM^+(\Gw,\gd^s) $ if $pq<1$. This solution satisfies 
	$$ \underline u_\mu \geq \BBG_s[\mu], \quad \underline v_\nu \geq \BBG_s[\nu] \quad \text{a.e. in } \Omega. $$
	
	Moreover, it is the minimal positive weak solution of \eqref{eq:system} in the sense that if $(u,v)$ is a positive weak solution of \eqref{eq:system} then $\underline u_\mu \leq u$ and $\underline v_\nu \leq v$ a.e. in $\Gw$.
	
	In addition, if $q<N_s$ then there exists a positive constant $K=K(N,s,p,q, \|\mu\|_{\GTM(\Gw,\gd^s)},\|\nu\|_{\GTM( \Gw,\gd^s)})$ such that $K \to 0$ as $(\|\mu\|_{\GTM( \Gw,\gd^s)},\|\nu\|_{\GTM(\Gw,\gd^s)}) \to (0,0)$ and
	\begin{equation} \label{leub}
	\max\{ \underline u_\mu, \underline v_\nu \} \leq K \BBG_s[\tilde\mu+\tilde\nu] \quad \text{a.e. in } \Omega,
	\end{equation}
	where $\tl\mu=\frac{\mu}{\|\mu\|_{\GTM(\Gw,\gd^s)}}$ and $\tl\nu=\frac{\nu}{\|\nu\|_{\GTM(\Gw,\gd^s)}}$.
\end{theorem}

The existence of the second solution is stated in the following theorem.
 
\begin{theorem} (Second solution) \lab{2nd sol}
	Assume $0\leq \mu,\, \nu\in L^r(\Omega)$ for some $r>\frac{N}{2s}$ and   $1<p\leq q<N_s$, where $N_s$ is defined in \eqref{p-s}. There exists $t^*>0$ such that if $$\max\{\|\mu\|_{L^r(\Gw)},\|\nu\|_{L^r(\Gw)} \} < t^*$$ 
	then system \eqref{eq:system} %with {\color{red} $\mu=\rho f$ and $\nu=\tau f$} 
	admits at least two positive weak solutions $(u_\mu,v_\nu)$ and $(\underline u_\mu, \underline v_{\nu})$ with  $u_\mu  \gneq \underline u_{\mu}$ and $v_\nu \gneq \underline v_{\nu}$, where $(\underline u_\mu, \underline v_{\nu})$ is the minimal solution constructed in Theorem \ref{existcoup}. 
	
	If, in addition, $\mu,\nu \in L^r(\Omega) \cap L_{loc}^\infty(\Gw)$ then $u_\mu  > \underline u_{\mu}$ and $v_\nu > \underline v_{\nu}$ in $\Gw$.
\end{theorem}

It is worth mentioning that the existence results in Theorem \ref{existcoup} and Theorem \ref{2nd sol} rely on completely different methods. More precisely, the proof of Theorem \ref{existcoup} is in spirit of \cite{BVY}, based on a delicate construction of a supersolution. The main ingredient is  a series of estimates concerning the Green kernel (see Lemmas \ref{3g}, \ref{ingg}, \ref{inggs} and  \ref{G3}). Theorem \ref{2nd sol} is obtained by using a variational approach. Because of the interplay of the two components $u$ and $v$, the analysis of the associated energy functional (see \eqref{I}) becomes complicated and consequently the Mountain-pass theorem is inapplicable. Therefore, we employ the Linking theorem instead. In order to construct the second solution of \eqref{eq:system}, we require the data $\mu$ and $\nu$ to be sufficiently regular, namely $\mu,\nu \in L^r(\Gw)$ for some $r>\frac{N}{2s}$. This enables to deduce the boundedness of the minimal solution constructed in Theorem \ref{existcoup}, which in turn allows to establish the geometry of the Linking theorem. As a consequence, we are able to prove the existence of a variational solution which is in fact a weak solution due to a result in \cite{A} and greater than the minimal solution.

The rest of the paper is organized as follows.  In Section 2, we collect some known estimates on the Green kernel from different papers and prove important estimates regarding the Green operator (see Lemmas \ref{ingg}, \ref{inggs}, \ref{G3}). These estimates are the main ingredient in the proof of Theorem \ref{existcoup} which is presented in Section 3.  In Section 4, we discuss a priori estimates, as well as regularity properties, of weak solutions.  Section 5 deals with the proof of Theorem \ref{2nd sol} which is based on the Linking theorem. 

\smallskip

{\bf Notations:} Throughout the present paper, we denote by $c,c',c_1,c_2,C,...$ positive constants that may vary from line to line. If necessary, the dependence of these constants will be made precise.

\section{Estimates on Green kernel}
%\begin{definition}
%	We say $u\in L^1(\Rn, \om)$ is $s$-harmonic in the  distributional sense if 
%	$$\Iom u(-\De)^s\phi=0 \quad\forall\phi\in C^{\infty}_0(\Om).$$
%\end{definition}
%
%In the next remark, we recall a result from \cite[Proposition 2.2]{BN} (also see \cite[Corollary 3.10 and Theorem 3.12]{BB}  for more details).
%\begin{remark}\label{harmonic}
%	\emph{(i)} $u$ is $s$-harmonic in the  distributional sense if and only if $u$ is $s$-harmonic in $\Gw$ in probabilistic sense.
%	
%	\emph{(ii)} $u$ is singular $s$-harmonic  in $\Gw$ in \emph{probabilistic sense} if and only if $u$ is  $s$-harmonic in $\Gw$ in the sense of distributions and $u=0$ in $\Gw^c$. 
%	
%\end{remark}

We denote by $G_s$ the Green kernel of $(-\De)^s$ in $\Om$ respectively. More precisely, for every $y\in\Om$,
\begin{equation}\left\{ \begin{aligned}
(-\Delta)^s G_s(.,y) &= \de_y \quad &&\text{in } \Omega \\
G_s(.,y) &=0  \quad &&\text{in } \Omega^c,
\end{aligned} \right.
\end{equation}
where $\de_y$ is the Dirac mass at $y$. 

\begin{lemma} (\cite[Corollary 1.3]{ChSo1}) \label{estGM0} There exists a positive constant $c=c(N,s,\Gw)$ such that
	\begin{equation} \label{G0} \begin{aligned} c^{-1}\min\{ |x-y|^{2s-N}, &\gd(x)^s \gd(y)^s |x-y|^{-N} \} \leq G_s(x,y) \\
	&\leq c \min\{ |x-y|^{2s-N}, \gd(x)^s \gd(y)^s |x-y|^{-N} \}, \quad \forall\, x \neq y, \, x,\,y \in \Gw. 
	\end{aligned} \end{equation}
\end{lemma}
\begin{lemma} (\cite[Theorem 1.1]{ChSo1}
	There exists a positive constant $C=C(N,s,\Gw)$ such that, for every $(x,y)\in\Gw\times\Gw,\;\;x\neq y$,
	\begin{align}\label{2.33}
	G_s(x,y)&\leq C \gd(y)^{s   }|x-y|^{s-N   },\\
	\label{2.34} G_s(x,y)&\leq C \frac{\gd(y)^{s   }}{\gd(x)^{s   }}|x-y|^{2s-N}.
	\end{align}
%	\be
%	G_s(x,y)\leq C\gd(x)^{s }\gd(y)^{s   }|x-y|^{-N   }.\label{2.35}
%	\ee
\end{lemma}
\begin{remark}\lab{l:22-6-1} Let $\theta \in (0,1]$, then there exists a positive constant $c=c(N,s,\theta,\Gw)$ such that $\BBG_s[\gl]^\theta \geq c \gd^s$ a.e. in $\Gw$ for every measure $\gl \in \GTM^+(\Gw,\gd^s)$ such that $\| \gl \|_{\GTM(\Gw,\gd^s)}=1$. Indeed, since $|x-y|<d_\Gw:=\text{diam}(\Gw)$, applying \eqref{G0} we have 
%it follows that $ |x-y|^{2s-N}>d_{\Gw}^{2s-N}$ and  $ |x-y|^{-N}>d_{\Gw}^{-N}$ for $x \neq y$.
%This and \eqref{G0} imply
$$ G_s(x,y) > c^{-1}\min\{ d_\Gw^{2s-N}, d_\Gw^{-N}\gd(x)^s\gd(y)^s  \}. $$	
Further, $\max\{\gd(x),\gd(y)\} < d_\Gw$ implies  %it follows that $\gd(x)^s \gd(y)^s < d_\Gw^{2s}$. Therefore, 
$ d_\Gw^{-N} \gd(x)^s \gd(y)^s < d_{\Gw}^{2s-N}.$	
Consequently, $G_s(x,y)>c^{-1} d_\Gw^{-N} \gd(x)^s \gd(y)^s$. It follows that
$\BBG_s[\lambda] \geq c\delta^s$ a.e. in $\Omega$. Therefore
$$ \BBG_s[\lambda]^\theta \geq c^\theta \gd^{\theta s} \geq c_1 \gd^s \quad \text{a.e. in } \Omega. 
$$
%As a consequence
%$$ \BBG_s[\gl](x) > c^{-1}d_{\Gw}^{-N} \gd(x)^s \int_{\Gw}\gd(y)^s d\gl(y) = c^{-1}d_{\Gw}^{-N} \gd(x)^s. $$
\end{remark}

\begin{remark} \label{G1} Let $\theta \in (0,1]$, then there exists a constant $c=c(N,s,\theta,\Gw)$ such that $\BBG_s[\gl]^\theta \geq c \BBG_s[1]$ in $\Gw$ for every measure $\gl \in \GTM^+(\Gw,\gd^s)$ such that $\| \gl \|_{\GTM(\Gw,\gd^s)}=1$. Indeed, by \cite[(2.18)]{CV}, we have $c_1^{-1}\gd^s \leq \BBG_s[1] \leq c_1 \gd^s$ in $\Gw$ for some constant $c_1=c_1(N,s,\Gw)$. This, combined with Remark \ref{l:22-6-1}, leads to the desired estimate.
\end{remark}
	\begin{definition}(Marcinkiewicz  space)
		Let $\Om\subset\Rn$ be a domain and $\la$ be a positive Borel measure in $\Om$. For $\kappa>1$, $\kappa'=\frac{\kappa}{1-\kappa}$ and $u\in L^1_{loc}(\Om, \la)$, we set
		$$\|u\|_{M^\kappa(\Om, \la)}:=\inf\displaystyle\bigg\{c\in[0,\infty]: \int_{E}|u|\,d\la \leq c \bigg(\int_{E}d\la\bigg)^\frac{1}{\kappa'}, \quad\forall\, E\subset\text{Borel set}\bigg\}$$
		and
		$$M^\kappa(\Om, \la):=\big\{u\in L^1_{loc}(\Om, \la): \|u\|_{M^\kappa(\Om, \la)}<\infty\big\}.$$
		$M^\kappa(\Om, \la)$ is called the Marcinkiewicz  space  with  exponent $\kappa$ (or weak $L^{\kappa}$ space) with  quasi-norm $\|.\|_{M^\kappa(\Om, \la)}$.
	\end{definition}
	The next lemma establishes a relation between Lebesgue space norm and Marcinkiewicz quasi-norm.
	
	\begin{lemma} (\cite[Lemma A.2(ii)]{BBC}) \lab{LpMk}
		Assume $1\leq q<\kappa<\infty$ and $u\in L^1_{loc}(\Om, \la)$. Then there exists $C(q,\kappa)>0$ such that for any Borel subset $E$ of $\Om$
		$$\int_{E}|u|^q\, d\la \leq C(q,\kappa)\|u\|^q_{M^{\kappa}(\Om, \la)}\bigg(\int_{E}d\la\bigg)^{1-\frac{q}{\kappa}}.$$		
	\end{lemma}
	
	We set
	\begin{equation} \label{k} k_{\alpha,\gamma}:=\left\{ \begin{aligned}
	&\frac{N+\alpha}{N-2s+\gamma} \quad &&\text{if } \alpha<\frac{N\gamma}{N-2s} \\
	&\frac{N}{N-2s} &&\text{otherwise}. 
	\end{aligned}
	\right. 
	\end{equation}
	
	Estimates of Green operator are presented below.
	
	\begin{lemma} (\cite[Proposition 2.2]{CV})  \label{estGM} Let $\alpha, \gg \in [0,s]$ and $k_{\alpha,\gg}$ be  as in \eqref{k}. 
		There exists $c=c(N,s,\alpha,\gg,\Gw)>0$ such that
		\begin{equation} \label{estG1} \|\BBG_s[\la]\|_{M^{k_{\alpha,\gg}}(\Gw,\gd^\alpha)}
		\leq c\|\la\|_{\GTM(\Gw,\gd^\gg)}\quad \forall\, \la \in \GTM(\Gw,\gd^\gg).
		\end{equation}
	\end{lemma}
	
	\begin{lemma} (\cite[Proposition 1.4]{RS2}) \label{regularity2} (i) If $t>\frac{N}{2s}$ then there exists $c=c(N,s,t,\Gw)$ such that
		\begin{equation} \label{estG2} 
		\| \BBG_s[\la] \|_{C^{\min\{ s, 2s-\frac{N}{t}  \}   }(\Gw)} \leq c\| \la \|_{L^t(\Gw)} \quad \forall\, \la \in L^t(\Gw).	
		\end{equation}
		(ii) If $1<t<\frac{N}{2s}$ then there exists a constant $c=c(N,s,t)$ such that
		\begin{equation} \label{estG2} 
		\| \BBG_s[\la] \|_{L^{\frac{Nt}{N-2ts} }(\Gw)} \leq c\| \la \|_{L^t(\Gw)} \quad \forall\, \la \in L^t(\Gw).
		\end{equation}
	\end{lemma}

We recall the $3$G-estimates.

\begin{lemma} (\cite[Theorem 1.6]{ChSo1}) \label{3g}
There exists a positive constant $C=C(\Gw,s)$ such that
\begin{equation} \label{e3g}
\frac{G_s(x,y)G_s(y,z)}{G_s(x,z)}\leq C \frac{|x-z|^{N-2s}}{|x-y|^{N-2s}|y-z|^{N-2s}} ,\quad\forall\, (x,y,z)\in\Gw\times\Gw\times\Gw.
\end{equation}
\end{lemma}	

Next we will prove some important estimates concerning the Green operator which will be used in Section 3.

\begin{lemma} \label{ingg} Let $0< p<N_s$ and $\la \in\GTM^+(\Gw,\gd^{s  })$ such that $\| \la \|_{\GTM(\Gw,\gd^s)}=1$. Then there exists a constant $C=C(N,s,p,\Gw)>0$ such that 
	\begin{equation} \label{laG} 
	\BBG_s[\BBG_s[\la]^p] \leq C\BBG_s[\la] \quad \text{a.e. in } \Omega. 
	\end{equation}
\end{lemma}

\begin{proof} First, we consider the case $p > 1$. From Lemma \ref{estGM} and the embedding $M^{N_s}(\Omega,\gd^s) \subset L^p(\Omega,\gd^s)$, we deduce that  $\BBG_s[\la]  \in L^{p}(\Gw,\gd^s)$. We write
	$$\BBG_s[\la](y)=\int_\Gw G_s(y,z)d\la(z)=\int_\Gw \frac{G_s(y,z)}{\gd(z)^{s  }}\gd(z)^{s   }d\la(z).$$
	Therefore,  using H\"{o}lder inequality, we obtain 
	$$\BBG_s[\la](y)^p\leq \int_\Gw\left(\frac{G_s(y,z)}{\gd(z)^{s   }}\right)^p \gd(z)^{s   } d\la(z),$$
	as $\| \la \|_{\GTM(\Gw,\gd^s)}=1$. 
	Consequently,
	\begin{equation} \label{in11}
	\BBG_s[\BBG_s[\la]^p](x)\leq \int_{\Gw}\int_{\Gw} G_s(x,y)G_s(y,z)^p\gd(z)^{s   (1-p)}d\la(z)dy.
	\end{equation}
	Now applying Lemma \ref{3g} and \eqref{2.33} to the right-hand side of the above expression, we obtain
	\begin{equation} \begin{aligned}
	&\int_{\Gw}\int_{\Gw} G_s(x,y)G_s(y,z)^p\gd(z)^{s (1-p)}d\la(z)dy\\ 
	&=\int_{\Gw}\int_{\Gw} G_s(x,y)G_s(y,z)\Big(\frac{G_s(y,z)}{\gd(z)^s}\Big)^{p-1}d\la(z)dy \\
	&\leq C_1\int_{\Gw}G_s(x,z)\int_{\Gw} \frac{|x-z|^{N-2s}}{|x-y|^{N-2s}|y-z|^{N-2s}}|y-z|^{-(N-s)(p-1)} dyd\la(z)\\
	&\leq C_2\int_{\Gw}G_s(x,z)\int_{\Gw}\big[|y-z|^{2s-N-(N-s)(p-1)}+|x-y|^{-(N-2s)}|y-z|^{-(N-s)(p-1)}\big]dyd\la(z)\\
%	&\leq C\int_{\Gw}G_s(x,z)\bigg[\int_{|x-y|\leq|y-z|}|x-y|^{-(N-2s)}|y-z|^{-(N-s)(p-1)} \\
%	&\qquad\qquad\qquad\qquad \qquad+\int_{|x-y|\geq|y-z|}|x-y|^{-(N-2s)}|y-z|^{-(N-s)(p-1)}\bigg]dyd\la(z)\\
	&\leq  C_3\int_{\Gw}G_s(x,z)\Big[\int_{\Omega \cap \{ |x-y| \geq |y-z| \}   }|y-z|^{2s-N-(N-s)(p-1)}dy\\
	&\qquad \qquad \qquad \quad + \int_{\Omega \cap \{ |x-y| \leq |y-z| \}   }|x-y|^{2s-N-(N-s)(p-1)}) dy \Big]d\la(z) \label{in12} \\
	&\leq C_4\int_{\Gw}G_s(x,z)d\la(z),
	\end{aligned} \end{equation}
	where $C_i=C_i(N,s,p,\Omega)$ ($i=1,2,3,4$). Here in the second estimate, we have used the inequality $|x-z| \leq |x-y| + |y-z|$ and in the last estimate we have used the fact that $p<N_s$.  Hence combining \eqref{in11} and \eqref{in12},  we derive that
	\begin{align*}
	\BBG_s[\BBG_s[\la]^p](x)
	&\leq C\int_{\Gw}G_s(x,z)d\la(z)=C\BBG_s[\la](x)
	\end{align*}
	where $C=C(N,s,p,\Omega)$. Note that the above argument is still valid for the case $p=1$.
	
	Next we consider the case $0< p<1$. Then we have
	$$ \BBG_s[\la]^p \leq C(p)(1+ \BBG_s[\la]) \quad \text{a.e. in } \Omega.
	$$
	This yields
	$$ \BBG_s[\BBG_s[\la]^p] \leq C(p)(\BBG_s[1] + \BBG_s[\BBG_s[\la]]) \quad \text{a.e. in } \Omega.
	$$
	By applying the case $p=1$, we have $ \BBG_s[\BBG_s[\la]] \leq C \BBG_s[\la]$ with $C=C(N,p,s,\Omega)$. Therefore, combining the above results along with Remark \ref{G1} with $\theta=1$, we derive \eqref{laG}. 
\end{proof}						

%%%%%%%%%%%%%%%%%%%%%%%%%%%

\begin{lemma} \label{inggs}
	Let $0< p<N_s$, $\la\in\GTM^+(\Gw,\gd^{s})$ with $\|\la\|_{\GTM(\Gw,\gd^{s})}=1$. Let  $\theta$ be such that
	\begin{equation} \label{s} \max\left(0,p-N_s+1\right)<\gth \leq 1. 
	\end{equation}
	Then there exists a positive constant $C=C(N,s,p,\theta,\Omega)$ such that
	\begin{equation} \label{ests} \BBG_s[\BBG_s[\la]^p ]\leq C \BBG_s[\la]^\gth \quad\text{ a.e. in } \;\Gw.
	\end{equation}
\end{lemma}
\begin{proof} First we assume that $p>1$. In view of the proof of Lemma \ref{ingg}, we have
	\begin{equation} \label{33a} \begin{aligned}
	&\BBG_s[ \BBG_s[\la]^p](x)\leq C\int_{\Gw}\int_{\Gw} G_s(x,y)G_s(y,z)^p\gd(z)^{s   (1-p)}d\la(z)dy\\ 
	&=C\int_{\Gw}\int_{\Gw}G_s(x,y)^{1-\gth} G_s(x,y)^\gth G_s(y,z)^\gth\left(\frac{G_s(y,z)}{\gd(z)^s}\right)^{p-\gth}\gd(z)^{s  (1-\gth)}d\la(z)dy.
	\end{aligned} \end{equation}
	By \eqref{e3g} and the inequality $|x-z| \leq |x-y|+|y-z|$,  we have
	\begin{equation} \label{knm} G_s(x,y)^\gth G_s(y,z)^\gth \leq CG_s(x,z)^\gth (|x-y|^{(2s-N)\gth} + |y-z|^{(2s-N)\gth}).
	\end{equation}
	Combining \eqref{33a} and \eqref{knm} yields
	\begin{equation} \BBG_s[\BBG_s[\la]^p](x) \leq C\int_{\Gw}G_s(x,z)^\gth \gd(z)^{s(1-\gth)}\int_{\Gw} I_{x,z}(y) dy d\la(z)
	\end{equation} 
	where
	$$ I_{x,z}(y) = \int_{\Gw}G_s(x,y)^{1-\gth}\left(\frac{G_s(y,z)}{\gd(z)^s}\right)^{p-\gth}(|x-y|^{(2s-N)\gth} + |y-z|^{(2s-N)\gth})dy.
	$$
	Applying \eqref{G0} and \eqref{2.33}, we obtain
	\begin{equation} \begin{aligned}
	I_{x,z}(y)& \leq C \int_{\Gw}|x-y|^{(2s-N)(1-\gth)}|y-z|^{(s-N)(p-\gth)}(|x-y|^{(2s-N)\gth} + |y-z|^{(2s-N)\gth})dy\\ 
	&\leq C  \int_{  \{y \in \Gw: |x-y| \leq |y-z|  \}   } |x-y|^{2s-N+(s-N)(p-\gth)}dy \\
	&+ C \int_{  \{y \in \Gw: |x-y| \geq |y-z|  \}   } |y-z|^{2s-N+(s-N)(p-\gth)}dy 
	\\
	&\leq C.
	\end{aligned} \end{equation}
	Here in the last inequlaity we have used the fact that $\theta> p-N_s+1$. Thus
	\begin{equation} \begin{aligned}
	\BBG_s[\BBG_s[\la]^p](x) &\leq C\int_{\Gw}\left(\frac{G_s(x,z)}{\gd(z)^{s   }}\right)^\gth \gd(z)^{s   }d\la(z)\\
	&\leq C\left(\int_{\Gw}G_s(x,z)d\la(z)\right)^\gth,
	\end{aligned} \end{equation}
	where in the last line we have used H\"{o}lder inequality with exponent $\frac{1}{\theta}$.	
	Note that the above approach is still valid for the case $p=1$.
	
	If $0< p<1$ then
	$$\BBG_s[ \BBG_s[\la]^p]\leq C(\BBG_s[1]+\BBG_s[\BBG_s[\la]])\leq C(\BBG_s[1]+\BBG_s[\la]^\gth).$$
	Then \eqref{ests} follows by a similar argument as in the proof of Lemma \ref{ingg} by using Remark \ref{G1} with $\theta \leq 1$. 
\end{proof}

In the sequel, without loss of generality, we may assume that 
\begin{equation} \label{p<q} 0<p \leq q.
\end{equation}
Hence, if $pq \geq 1$ then 
\begin{equation}\label{pqrel} p \leq q\frac{p+1}{q+1} \leq p\frac{q+1}{p+1} \leq q. \end{equation}
Put
$$t_s: =q\left(p- N_s + 1\right). $$
Notice that if $q\frac{p+1}{q+1} <N_s$ then $t_s <q\frac{p +1}{q+1}<N_s$.
\begin{lemma} \label{G3}
	Let $p,q>0$, $p\leq q$ and $\la\in\GTM^+(\Gw,\gd^{s})$ with $\|\la\|_{\GTM(\Gw,\gd^{s})}=1$.  Assume  $q\frac{p +1}{q+1} <N_s$. Then for any $t \in (\max(0,t_s),q]$, there exists a positive constant $c=c(N,p,q,s,t)$  such that
	\begin{equation} \label{G31} \BBG_s[\BBG_s[\la]^p]^{q} \leq c \BBG_s[\la]^t \quad \text{a.e. in } \Omega.
	\end{equation}
	In particular,
	\begin{equation} \label{G32} \BBG_s[ \BBG_s[\la]^p]^{q} \leq C \BBG_s[\la]^{q\frac{p +1}{q+1}} \quad \text{a.e. in } \Omega, 
	\end{equation}
	\begin{equation} \label{G33} \BBG_s[ \BBG_s[\BBG_s[\la]^p]^{q}] \leq C\BBG_s[\la] \quad \text{a.e. in } \Omega,
	\end{equation}
	where $C=C(N,p,q,s)$.
\end{lemma}

\begin{proof} Since $p\leq q$ from \eqref{p<q}, it follows that $p\leq q\frac{p+1}{q+1}$. Therefore, from the assumption, it follows $p<N_s$. Hence $\max(0,p-N_s+1) < 1$. Let $t \in (\max(0,t_s),q]$ then $\max(0,p-N_s+1) <\frac{t}{q} \leq 1$. Therefore, applying Lemma \ref{inggs} with $\gth$ replaced by $\frac{t}{q}$, we obtain
	$$ \BBG_s[\BBG_s[\la]^p] \leq c \BBG_s[\la]^{\frac{t}{q}},
	$$
	which implies \eqref{G31}. 
	
	Since $t_s<q\frac{p+1}{q+1}  \leq q$, taking $t=q \frac{p+1}{q+1}$ in \eqref{G31} yields \eqref{G32}. Next, since $q\frac{p+1}{q+1} < N_s$, combining \eqref{G32} along with Lemma \ref{ingg}, we have
	$$ \BBG_s[\BBG_s[ \BBG_s[\la]^p]^{q}] \leq C\BBG_s[\BBG_s[\la]^{q\frac{p+1}{q+1}} ] \leq C\BBG_s[\la] \quad \text{a.e. in } \Omega.
	$$
	This completes the proof.
	\end{proof}

\section{Construction of the minimal solution}

\begin{lemma} \label{compa}  Assume  %$p, q \in (1, N_s)$ and $N_s$ is defined in \eqref{p-s} 
$p,q>0$ and $\mu,\, \nu \in \GTM^+(\Gw,\delta^s)$. Assume in addition that there exist functions $V \in L^p(\Gw,\gd^s)$ and $U \in L^q(\Gw,\gd^s)$ such that 
\begin{equation} \label{UV} \begin{aligned}U &\geq \BBG_s[V^p] + \BBG_s[\mu] \quad \text{a.e. in } \Omega,\\
V &\geq \BBG_s[U^q] + \BBG_s[\nu] \quad \text{a.e. in } \Omega.
\end{aligned}\end{equation} 
Then there exists a positive minimal weak solution $(\underline u_\mu, \underline v_\nu)$ of  \eqref{eq:system} satisfying
	\begin{equation}\label{compa1} \BBG_s[\mu] \leq \underline u_\mu \leq U,\quad \BBG_s[\nu] \leq \underline v_\nu \leq V  \quad \text{a.e. in } \Omega.
\end{equation}
\end{lemma}
\begin{proof}
	Put $u_0:=\BBG_s[\mu] $, $v_0:=\BBG_s[\nu] $ and for $n\geq 1$, define
	\begin{equation}\begin{aligned} \label{induc}  u_n& :=\BBG_s[v_{n-1}^p]+ \BBG_s[\mu], \\
	 v_n& :=\BBG_s[u_{n-1}^q]+ \BBG_s[\nu].
	\end{aligned}\end{equation}	
	Clearly $u_0 \leq U$ and $v_0\leq V$ . Therefore,
	$$ u_1=\BBG_s[v_0^p] + \BBG_s[\mu] \leq \BBG_s[V^p] + \BBG_s[\mu] \leq U. $$ 
Similarly, $v_1\leq \BBG_s[U^q] + \BBG_s[\nu]\leq V$.
By induction, it follows that $u_n \leq U$ and $v_n\leq V$ for every $n \geq 1$. Also, it is easy to see that $\{u_n\}$ and $\{v_n\}$ are increasing sequences. Hence  $u_n \uparrow \underline u_\mu \leq U \in L^q(\Gw,\gd^s)$ and
$v_n \uparrow \underline v_\nu \leq V \in L^p(\Gw,\gd^s)$.
Therefore $\BBG_s[v_n^p] \uparrow \BBG_s[\underline v_\nu^p]$ and
$\BBG_s[u_n^q] \uparrow \BBG_s[\underline u_\mu^q]$
 a.e. in $\Gw$. Letting $n \to \infty$ in \eqref{induc}, we deduce that 
	$$ \underline u_\mu = \BBG_s[\underline v_\nu^p] + \BBG_s[\mu] \quad\text{and}\quad \underline v_\nu = \BBG_s[\underline u_\mu^q] + \BBG_s[\nu]. $$
	This means that $(\underline u_\mu, \underline v_\nu)$ is a weak solution of \eqref{eq:system}.  
	
Next, let $(u, v)$ be any positive weak solution \eqref{eq:system}. Then 
	$$u=\BBG_s[v^p]+\BBG_s[\mu] \geq u_0,\quad v=\BBG_s[u^q]+\BBG_s[\nu] \geq v_0.$$
Thus $$u\geq \BBG_s[v_0^p]+\BBG_s[\mu] \geq u_1, \quad v\geq \BBG_s[u_0^q]+\BBG_s[\nu] \geq v_1.$$
	By induction it follows that $u\geq u_n$ and $v\geq v_n$ for all $n\geq 1$. Hence $u\geq \underline u_\mu$ and $v\geq \underline v_\nu$. This completes the lemma.
\end{proof}
\begin{remark}\lab{l:3-9-1}
In stead of studing system \eqref{eq:system}, in the sequel, we will work on the following system 
\begin{equation} \label{source-3} \left\{ \begin{aligned}
(-\Delta)^s u &= v^p + \rho\mu \quad &&\text{in } \Omega, \\
(-\Delta)^s v &= u^q + \tau \nu \quad &&\text{in } \Omega, \\
u =v&=0 \quad &&\text{in } \Omega^c,
\end{aligned} \right.
\end{equation}
where $\rho, \tau$ are positive parameters and $\mu,\nu \in \GTM^+(\Omega,\gd^s)$ such that $\| \mu \|_{\GTM(\Omega,\gd^s)}=\| \nu \|_{\GTM(\Omega,\gd^s)}=1$. The advantage is when dealing with system \eqref{source-3}, we can easily apply Lemma \ref{inggs} and Lemma \ref{G3} for $\BBG_s[\mu]$ and $\BBG_s[\nu]$ and  require only the smallness of the parameters $\rho$ and $\tau$, which improves considerably the exposition.
\end{remark}

We recall that in the sequel, we assume that $0<p \leq q$ and hence if $pq \geq 1$ then \eqref{pqrel} holds. 
 
\begin{theorem} \label{existcoup-1}
Let $p,\, q,\,\rho,\, \tau >0$ and $\mu, \nu \in \GTM^+(\Gw,\delta^s)$ such that $\| \mu \|_{\GTM(\Omega,\gd^s)}=\| \nu \|_{\GTM(\Omega,\gd^s)}=1$. Assume $p \leq q$, $pq \neq 1$, $q\frac{p+1}{q+1}<N_s$ and $\BBG_s[\mu]\in L^{q}(\Gw,\gd^{s})$. Then system \eqref{source-3}
admits the minimal weak solution $(\underline u_{\rho \mu},\underline v_{\tau \nu})$ for $\rho$ and $\tau$ small if $pq>1$, for any $\rho>0$ and $\tau>0$ if $pq<1$. 

In addition, if $p \leq q<N_s$, then
there exists a constant $K=K(N,s,p,q,\Gw,\rho,\tau)$ such that $K \to 0$ as $(\rho,\tau)\to (0,0)$ and
\begin{equation} \label{leub}
\max\{ \underline u_{\rho\mu},\underline v_{\tau \nu} \} \leq K \BBG_s[ \mu+  \nu] \quad \text{a.e. in } \Omega.
\end{equation}
\end{theorem}
\begin{proof} Fix numbers $\vartheta_i>0$, ($i=1,2$) and set
	\begin{equation} \label{Psi} \Psi:=\BBG_s[\vartheta_1 \mu]^{q} + \vartheta_2 \nu.
	\end{equation}
	For $\gk \in(0,1]$, put
	\begin{equation} \label{rt} \rho= \gk^{\frac{1}{q}}\vartheta_1, \quad \tau= \gk \vartheta_2
	\end{equation}
	and consider system \eqref{source-3} with $\rho$ and $\tau$ as in \eqref{rt}. From the assumptions, $\BBG_s[\mu]^q \in L^1(\Omega,\gd^s)$ and $\nu \in \GTM^+(\Gw,\gd^s)$, it follows that  $\Psi \in \GTM^+(\Gw,\gd^{s  })$. Also we note that $p\leq q$ and $q\frac{p+1}{q+1}<N_s$ imply $p<N_s$.
Set
$$ V:=A\BBG_s[\gk\Psi] \quad \text{and} \quad U:=\BBG_s[ V^p] + \BBG_s[ \rho \mu ]$$
where $A>0$ will be determined later on. Then, 
\begin{align*} U^q  &\leq c (  \BBG_s [V^p]^q + \BBG_s[\rho \mu]^q)  \\
&= c \Big\{  A^{p q} \gk^{pq} \BBG_s[ \BBG_s[\Psi]^p]^{q} +   \BBG_s[\rho \mu]^{q}\Big\}
\end{align*}
where $c=c(p,q)$. It follows that
\begin{equation} \label{lat} \begin{aligned}
\BBG_s[U^q]+\BBG_s[\tau \nu] &\leq c (A^{p q} \gk^{pq} \BBG_s[ \BBG_s[ \BBG_s[\Psi]^p]^{q}] + \BBG_s[\BBG_s[\rho \mu]^q]) + \BBG_s[\tau \nu] \\
&= c (A^{p q} \gk^{pq} \BBG_s[ \BBG_s[ \BBG_s[\Psi]^p]^{q}] + \kappa \BBG_s[\BBG_s[\vartheta_1 \mu]^q]) + \gk \BBG_s[\vartheta_2 \nu] \\
&\leq c (A^{p q} \gk^{pq} \BBG_s[ \BBG_s[ \BBG_s[\Psi]^p]^{q}] +  \gk \BBG_s[\Psi]).
\end{aligned} \end{equation}
Since  $q\frac{p+1}{q+1} <N_s$, applying Lemma \ref{G3}, we have
\begin{equation} \label{laca}\BBG_s[\BBG_s[\BBG_s[\Psi]^p]^{q}] \leq C\| \Psi \|_{\GTM(\Gw,\gd^s)}^{pq-1} \BBG_s[\Psi]
\end{equation}
where $C=C(N,s,p,q,\Gw)$. Combining the definition of $\Psi$ in \eqref{Psi}, Lemma \ref{G3}, Lemma \ref{estGM} and the assumption that $\| \mu \|_{\GTM(\Gw,\gd^s)}=\| \nu \|_{\GTM(\Gw,\gd^s)}=1$, we can estimate
$ C^{-1} \leq \| \Psi \|_{\GTM(\Gw,\gd^s)}  \leq C $
for some positive constant $C$ independent of $A$ and $\kappa$. This, combined with \eqref{lat} and \eqref{laca} implies
\begin{equation} \label{leni} \BBG_s[U^q]+\BBG_s[\tau \nu]  \leq C(A^{pq}\gk^{pq} + \gk)\BBG_s[\Psi].
\end{equation}
for some positive constant $C$ independent of $A$ and $\kappa$.

We will choose $A$ and $\gk$ such that
\begin{equation} \label{GU<V} \BBG_s[U^q]+\BBG_s[\tau \nu]  \leq V.
\end{equation}
For that, it is sufficient to choose $A$ and $\gk$ such that
$$ C(A^{pq}\gk^{pq} + \gk)\BBG_s[\Psi] \leq V.
$$
This holds if 
\begin{equation} \label{rl} C(A^{pq}\gk^{pq-1} + 1) \leq A. 
\end{equation}
If $p q>1$ then we can choose $A>0$ large enough and then choose  $\gk>0$ small enough (depending on $A$) such that \eqref{rl} holds. If $p q<1$ then for any $\gk>0$ there exists $A$ large enough such that \eqref{rl} holds. For such $A$ and $\gk>0$, we obtain  \eqref{GU<V}. Consequently, $(U,V)$ satisfies \eqref{UV}.  By Lemma \ref{compa}, there exists a weak solution $(\underline u_{\rho \mu}, \underline v_{\tau \nu})$ of \eqref{source-3} for $\rho>0$ and $\tau>0$ small if $p q>1$, for any $\rho>0$ and $\tau>0$ if $p q<1$. Moreover, $(\underline u_{\rho \mu}, \underline v_{\tau \nu})$ satisfies \eqref{compa1}. 

Next, assuming in addition that $p \leq q<N_s$, we will demonstrate \eqref{leub}. From the definition of $V$ and Lemma \ref{inggs}, we see that
\begin{equation} \label{V1} \begin{aligned} V &=A\kappa\BBG_s[\Psi] = A\kappa(\BBG_s[\BBG_s[\vartheta_1 \mu]^q] + \BBG_s[\vartheta_2 \nu]) \\
& \leq c A\kappa \BBG_s[\mu + \nu].
\end{aligned} \end{equation}

It follows that
$$ V^p \leq C \kappa^p( \BBG_s[\mu]^p + \BBG_s[\nu]^p ).
$$
Therefore, 
\begin{equation} \label{U1} \begin{aligned} U &\leq C \kappa^p (\BBG_s[\BBG_s[\mu]^p] + \BBG_s[\BBG_s[\nu]^p]) + \kappa^\frac{1}{q} \BBG_s[\vartheta_1 \mu] \\
&\leq C\kappa^p \BBG_s[\mu + \nu] + \kappa^\frac{1}{q}\vartheta_1 \BBG_s[ \mu] \\
&\leq C\max\{ \kappa^p,\kappa^\frac{1}{q} \} \BBG_s[\mu + \nu].
\end{aligned} \end{equation}
Combining \eqref{U1} and \eqref{V1} along with \eqref{compa1} leads to \eqref{leub}.
\end{proof}
%\begin{remark} \label{rhotau0} 
%It is easy to see from the proof of  Theorem \ref{existcoup} that there exist %$\varrho_0>0$ and $\tau_0>0$ such that if $\| \mu \|_{\GTM(\Omega)} < \varrho_0$ %and  $\|  \nu\|_{\GTM(\Omega)} < \tau_0$ then system Theorem \ref{source-2} %admits a minimal solution $(\underline u_\mu, \underline v_\nu)$.
%\end{remark}

\begin{remark} \label{Linf} By Lemma \ref{regularity2}(i), we see that if $\mu, \nu \in L^r(\Gw)$ with $r>\frac{N}{2s}$ then $\BBG_s[\mu + \nu] \in L^\infty(\Gw)$. From  Theorem \ref{existcoup-1}, if $\rho$ and $\tau$ are small  then system \eqref{source-3} admits a minimal solution $(\underline u_{\rho \mu}, \underline v_{\tau \nu})$ which satisfies \eqref{leub}. It follows that $\underline u_{\rho \mu}, \underline v_{\tau \nu} \in L^\infty(\Gw)$. Moreover, $(\underline u_{\rho \mu}, \underline v_{\tau \nu}) \to (0,0)$  a.e. in $\Omega$ as $(\rho,\tau) \to (0,0)$.	
\end{remark}

\vspace{2mm}

{\bf Proof of Theorem \ref{existcoup} completed}: Combining Theorem \ref{existcoup-1} along with Remark \ref{l:3-9-1}, the proof of the theorem follows.

\section{A priori estimates and regularity}
In this section, we provide a priori estimates, as well as regularity properties, of weak solutionsof \eqref{eq:system}.
\subsection{A priori estimates}
\begin{lemma} \label{apriori} Assume   $p>1, \,q> 1$ and $\mu,\, \nu \in \GTM^+(\Gw,\gd^s)$. If $(u,v)$ is a weak solution of \eqref{eq:system} then there is a positive constant $c=c(N,s,p,q,\Gw)$ such that
	\begin{equation}\begin{aligned} \label{rem1} \norm{u}_{L^1(\Gw)} + \norm{v}_{L^p(\Gw,\gd^s)}& \leq c(1+\norm{\mu}_{\GTM(\Gw,\gd^s)}),\\
	\norm{v}_{L^1(\Gw)} + \norm{u}_{L^q(\Gw,\gd^s)} &\leq c(1+\norm{\nu}_{\GTM(\Gw,\gd^s)}). 
	\end{aligned}
	\end{equation}
\end{lemma}
\begin{proof}	
We prove this lemma in the spirit of \cite[Lemma 4.1]{BN}. Let $(\la_1,\vgf_1)$ be the first eigenvalue and corresponding positive eigenfunction of $(-\De)^s$ in $X_0(\Om)$ (see the definition of $X_0$ in \eqref{eq:X0}). By \cite[Lemma 2.1(ii)]{CV}, $\vgf_1\in \BBX_s(\Om)$, and hence by taking $\zeta=\vgf_1$ in \eqref{intN}, we obtain
\begin{equation}\begin{aligned} \label{rer} \gl_1 \int_{\Gw} u\vgf_1 dx&= \int_{\Gw} v^p \vgf_{1} dx  +\gl_1\int_{\Gw} \vgf_1 d\mu ,\\
\gl_1 \int_{\Gw} v\vgf_1 dx&= \int_{\Gw} u^q \vgf_{1} dx  +\gl_1\int_{\Gw} \vgf_1 d\nu .
 \end{aligned}
	\end{equation}
	By Young's inequality, we get
	\begin{equation}\begin{aligned}  \label{rem2} 
	\int_{\Gw} v\vgf_1 dx &\leq (2\gl_1)^{-1}\int_{\Gw} v^p\vgf_{1} dx + (2\gl_1)^{\frac{1}{p-1}}\int_{\Gw}\vgf_1dx,\\
	\int_{\Gw} u\vgf_1 dx &\leq (2\gl_1)^{-1}\int_{\Gw} u^q\vgf_{1} dx + (2\gl_1)^{\frac{1}{q-1}}\int_{\Gw}\vgf_1dx.
	 \end{aligned}\end{equation}	
Substituting \eqref{rem2} in \eqref{rer} we have
\begin{equation} \begin{aligned} \label{rem3} 2\int_{\Gw} v^p \vgf_1 dx  +
	2\gl_1 \int_{\Gw} \vgf_1 d\mu &\leq \int_{\Om} u^q\vgf_1 dx + (2\gl_1)^{\frac{q}{q-1}}\int_{\Gw}\vgf_1 dx,\\
	2\int_{\Gw} u^q \vgf_1 dx  +
	2\gl_1 \int_{\Gw} \vgf_1 d\nu &\leq \int_{\Om} v^p\vgf_1 dx +(2\gl_1)^{\frac{p}{p-1}}\int_{\Gw}\vgf_1 dx.
 \end{aligned}\end{equation}	
Therefore,
$$
\frac{3}{2}\int_{\Gw} v^p \vgf_1 dx  +\la_1\int_{\Gw} \vgf_1 d(2\mu+\nu)\leq \bigg(\frac{(2\la_1)^\frac{p}{p-1}}{2}+ (2\la_1)^\frac{q}{q-1}\bigg)\int_{\Om}\vgf_1 dx.
$$
Since the second term on the left-hand side of above expression is nonnegative, taking into account that $c^{-1}\delta^s < \vgf_1 < c\delta^s$ in $\Omega$, we have
	\begin{equation} \label{rem4} \norm{v}_{L^p(\Gw,\gd^s)}^p \leq C(\la_1)\int_{\Gw}\gd^s dx \leq c'.
	\ee
Similarly \begin{equation} \label{rem4'} \norm{u}_{L^q(\Gw,\gd^s)}^q \leq C'.\ee
Next, combining \eqref{uGM}, Lemma \ref{LpMk} and Lemma \ref{estGM} with $\gamma=s,\, \al=0$ we obtain
	\begin{equation}\begin{aligned}  \label{rem5} 
	\|u\|_{L^1(\Om)} &\leq \|\BBG_s[v^p]\|_{L^1(\Om)}+\|\BBG_s[\mu]\|_{L^1(\Om)} \\ 
	&\leq 	C\big(\|\BBG_s[v^p]\|_{M^{N_s}(\Om)}+\|\BBG_s[\mu]\|_{M^\frac{N}{N-s}(\Om)}\big)\\
	&\leq C\big( \|v\|_{L^p(\Om, \de^s)}^p+\|\mu\|_{\mathfrak{M}(\Om,\gd^s)}\big).
	\end{aligned}\end{equation}
Hence first expression of \eqref{rem1} holds by combining \eqref{rem4} and \eqref{rem5}. Similarly, the second expression of \eqref{rem1} follows.	
\end{proof}

\subsection{Regularity}
\begin{theorem} \label{reg} Let $p,q \in (1,N_s)$.  
(i) Assume $\mu,\nu \in L^r(\Gw)$ for some $r>\frac{N}{2s}$. If $(u,v)$ is a nonnegative weak solution of \eqref{eq:system} then $u,v \in L_{loc}^\infty(\Gw)$.	

(ii) Assume $\mu,\nu \in L^r(\Gw) \cap L_{loc}^\infty(\Gw)$. If $(u,v)$ is a nonnegative weak solution of \eqref{eq:system} then $u,v \in C_{loc}^{\alpha}(\Gw)$ for some $\alpha \in (0,2s)$.	
\end{theorem}
\begin{proof} (i) We first assume that $\mu,\nu \in L^r(\Gw)$ for some $r>\frac{N}{2s}$. Let $(u, v)$ be a  nonnegative weak solution of \eqref{eq:system}. Then $u \in L^1(\Om)\cap L^q(\Gw,\gd^s)$, $v \in L^1(\Om)\cap L^p(\Gw,\gd^s)$ and $(u,v)$ satisfies \eqref{uGM}. Let $x_0 \in \Gw$ and $r>0$ such that $B(x_0,2r) \subset\subset  \Gw$. For any $j \in \BBN$, set $B_j:=B(x_0,2^{-j}r)$. For any $j \in \BBN$, we can write 
	\begin{equation}\begin{aligned} \label{decomp} 
	u &= \BBG_s[\chi_{\Gw \setminus B_j}v^p] + \BBG_s[\chi_{B_j}v^p] + \BBG_s[\mu],\\ 
	v &=\BBG_s[\chi_{\Gw \setminus B_j}u^q] + \BBG_s[\chi_{B_j}u^q] + \BBG_s[\nu].
	\end{aligned}\end{equation} 
	
	Observe that, for $x \in B_{j+1}$, by \eqref{G0},
	\begin{align*}
	\BBG_s[\chi_{\Gw \setminus B_j}v^p](x) = \int_{\Gw \setminus B_j} v(y)^pG_s(x,y)dy 
	&\leq C\gd(x)^s  \int_{\Gw \setminus B_j} v(y)^p\gd(y)^s |x-y|^{-N}dy \\
	&\leq C 2^{(j+1)N} r^{-N} \| v \|_{L^p(\Gw,\gd^s)}^p<\infty.
	\end{align*}
	Therefore, 
	\begin{equation} 
	\label{lubu} \BBG_s[\chi_{\Gw \setminus B_j}v^p] \in L^\infty(B_{j+1}) \quad \forall j \in \BBN. 
	\end{equation}
	Similarly,
	\begin{equation} \label{lubu-1} \BBG_s[\chi_{\Gw \setminus B_j}u^q] \in L^\infty(B_{j+1}) \quad \forall \, j \in \BBN. \end{equation}
	Since $\mu,\nu \in L^r(\Gw)$ for $r>\frac{N}{2s}$, by Lemma \ref{regularity2} (i), we deduce 
	\begin{equation} \label{Ginfty} \BBG_s[\mu], \, \BBG_s[\nu] \in L^\infty(\Gw). \end{equation}
	%Next, using Lemma \ref{LpMk} and  Lemma \ref{estGM}, we obtain for $\theta \in(1,N_s)$,
	%$$  \begin{aligned}\|u\|_{L^\theta(\Gw,\gd^s)} &\leq C(\|v\|_{L^p(\Gw,\gd^s)} +  \| \mu \|_{\GTM(\Gw,\gd^s)}) <C'\\
	%\|v\|_{L^\theta(\Gw,\gd^s)} &\leq C(\|u\|_{L^q(\Gw,\gd^s)} +  \| \nu \|_{\GTM(\Gw,\gd^s)}) <C'.\end{aligned}$$	
	%In particular, $u \in L^q(\Gw,\gd^s)$ and $v \in L^p(\Gw,\gd^s)$. Consequently 
	
Further, as  $u \in L^q(\Gw,\gd^s)$ and $v \in  L^p(\Gw,\gd^s)$, we have  $\chi_{B_0}u \in L^q(\Omega)$ and $\chi_{B_0}v \in L^p(\Omega)$ and therefore, applying Lemma \ref{estGM} and Lemma \ref{LpMk} we have
$$\|\BBG_s[\chi_{B_0}u^q]\|_{L^\theta(\Om)}\leq c\|\BBG_s[\chi_{B_0}u^q]\|_{M^\frac{N}{N-2s}(\Om)}\leq c'\|\chi_{B_0}u\|_{L^q(\Om)}^q,$$
for every $1<\theta<\frac{N}{N-2s}$. This in turn implies $\BBG_s[\chi_{B_0}u^q],\, 
	\BBG_s[\chi_{B_0}v^p] \in L^\theta(B_0)$ for every $1<\theta<\frac{N}{N-2s}$.  This and \eqref{decomp} -- \eqref{Ginfty} yield $u, v \in L^\theta(B_2)$ for every $1<\theta<\frac{N}{N-2s}$.  Set,
	$$\ell_0:=\frac{1}{2}(1+\frac{N_s}{p}),\quad \tilde \ell_0:=\frac{1}{2}(1+\frac{N_s}{q}).$$
	Then $1<p\ell_0, \, q\tilde \ell_0<N_s<\frac{N}{N-2s}$ and hence $u \in L^{q\tilde \ell_0}(B_{2})$ and $v\in L^{p\ell_0}(B_2)$.  Without loss of generality, we can assume $\ell_0,\, \tilde \ell_0\not=\frac{N}{2s}$. If  $\ell_0,\, \tilde \ell_0>\frac{N}{2s}$, then by Lemma \ref{regularity2} (i), $\BBG_s[\chi_{B_2}v^p], \, \BBG_s[\chi_{B_2}u^q] \in L^\infty(B_{2})$. This and \eqref{decomp} -- \eqref{Ginfty} imply $u, \, v \in L^\infty(B_4)$. If $\ell_0<\frac{N}{2s}$ or $\tilde \ell_0<\frac{N}{2s}$, then by Lemma \ref{regularity2} (ii) we obtain $\BBG_s[\chi_{B_2}v^p] \in L^{p\ell_1}(B_{2})$ or $\BBG_s[\chi_{B_2}u^q] \in L^{q\tilde \ell_1}(B_{2})$ respectively, where 
	$$ \ell_1:=\frac{1}{p}\frac{N\ell_0}{N-2\ell_0s}, \quad \tilde \ell_1:=\frac{1}{q}\frac{N\tilde \ell_0}{N-2\tilde \ell_0s}.
	$$
	Then from \eqref{decomp} -- \eqref{Ginfty}, $v \in L^{p\ell_1}(B_4)$ or $u \in L^{q \tilde \ell_1}(B_4)$. We have
	$$ \frac{\ell_1}{\ell_0}=\frac{1}{p}\frac{N}{N-2\ell_0s}>\frac{1}{p}\frac{N}{N-2s}>\ell_0.
	$$
	This implies that $\ell_1>\ell_0^2>\ell_0>1$. Similarly, $\tilde \ell_1>\tilde \ell_0^2>\tilde \ell_0>1$. Now if $\ell_1$ or $\tilde \ell_1\not=\frac{N}{2s}$, then continuing the bootstrap method as in the proof of \cite[Theorem 1.6]{BN}), we can conclude that 
	$u, \, v \in L^\infty(B_{2(k+1)})$. Consequently,  $u, \, v \in L_{loc}^\infty(\Gw)$.  
	\vspace{2mm}
	
	(ii) If $\mu,\nu \in L^r(\Gw)\cap L_{loc}^\infty(\Gw)$ then by part (i), we have  $v^p+\mu,\, u^q+\nu \in L^\infty_{loc}(\Om)$. Further, as $u, v\in L^1(\Om)$,  applying Schauder estimate \cite{RS1}, we have $u, v\in C^{\alpha}_{loc}(\Om)$, for some $\alpha\in(0,2s)$. 
	
	\end{proof}

\section{Construction of a second solution}

In this section we assume $1<p \leq q <N_s$. Then it follows that  $q\frac{p+1}{q+1}<N_s$. Using Linking theorem, we will construct a second weak solution of \eqref{source-3} when $\mu,\nu \in L^r(\Omega)$, for $r>\frac{N}{2s}$ with $\|\mu\|_{L^r(\Omega)}=\|\nu\|_{L^r(\Omega)}=1$.

%\begin{equation} \label{source-2"} \left\{ \begin{aligned}
%(-\Delta)^s u &= v^p \quad \text{in } \Omega, \\
%(-\Delta)^s v &= u^q \quad \text{in } \Omega, \\
%\tr_s(u) =\rho \ell, &\quad \tr_s(v) =\tau \ell, \\
%u =v&=0 \quad \text{in } \Omega^c,
%\end{aligned} \right.
%\end{equation}
%where $0\leq \ell\in L^\infty(\Rn)$, $\|\ell\|_{L^\infty(\Rn)}=1$, %$1<p<q<N_s$ and $q\frac{p+1}{q+1}<N_s$.

By Theorem \ref{existcoup-1}, if $\rho>0$ and $\tau>0$ are small then there exists the minimal positive week solution, denoted by $(\underline u_{\rho \mu}, \underline v_{\tau \nu})$,  of \eqref{source-3}.
 We would like to apply Linking theorem to find a variational weak solution of 
\begin{equation} \label{MP-prob} \left\{  
\begin{aligned}
(-\De)^s u &=(\underline v_{\tau \nu}+v^+)^p- \underline v_{\tau \nu}^p \quad &&\text{in }\Om\\
(-\De)^s v &=(\underline u_{\rho \mu}+u^+)^q- \underline u_{\rho \mu}^q \quad &&\text{in }\Om\\
u &= 0=v &&\text{in } \Om^c,
\end{aligned} \right.
\end{equation}
where $u^+:=\max(u, 0)$ and $u^-:=-\min(0, u)$. 

From Remark \ref{Linf}, we observe that  there exists a constant $M>0$ such that
\be\lab{29-5-2}\max\{\underline u_{\rho \mu}, \underline v_{\tau \nu} \}<M \quad \text{in } \Omega. \ee

Define
\be\label{eq:X0}
X_0:=\{w\in H^s(\Rn): w=0\quad \text{in}\quad \Rn\setminus\Om\},
\ee
where $H^s(\Rn)$ is the standard fractional Sobolev space on $\Rn$. It is well-known that 
\be\label{norm-X}
\|w\|_{X_0}:=\displaystyle\left(\int_{Q}\f{|w(x)-w(y)|^2}{|x-y|^{N+2s}}dxdy\right)^\f{1}{2},
\ee
where $Q=\R^{2N}\setminus (\Om^c\times\Om^c)$, is a norm on $X_0$ and $(X_0, ||.||_{X_0})$ is a Hilbert space, with the inner product 
$$\<\phi, \psi\>_{X_0}:=\int_{Q}\f{(\phi(x)-\phi(y))(\psi(x)-\psi(y))}{|x-y|^{N+2s}}dxdy.$$ 
Put $$2^*_s:=\frac{2N}{N-2s}.$$
It is easy to check that (see \cite{RS3}) 
$$\Iom \psi(-\De)^s \phi\, dx=\int_{\Rn}(-\De)^\frac{s}{2}\phi(-\De)^\frac{s}{2}\psi \, dx\quad \forall \phi, \psi\in X_0.$$
It is also well known that the embedding $X_0\hookrightarrow L^r(\Rn)$ is compact, for any $r\in[1, 2_s^*)$ 
and  $X_0\hookrightarrow L^{2_s^*}(\Rn)$ is continuous. 

%\begin{lemma}\lab{l:cpt}\cite[Proposition 2.3]{CQ}
%	Let $1<t<N_s$. Then the embedding $X_0 \hookrightarrow L^2\big(\Om, \frac{dx}{|x|^{(N-s)(t-1)}}\big)$ is continuous and compact.  
%\end{lemma}

\begin{definition} \label{defstable}
	We say that a solution $(u,v)$ of \eqref{eq:system} is stable (resp. semistable)  if 
	\begin{equation} \label{stabex} \left\{  \begin{aligned} \| \gf \|_{X_0}^2 > \,(\textrm{resp.\,}\geq)\,\,  p\int_{\Om}v^{p-1}\phi^2 dx, \\
	\| \gf \|_{X_0}^2 > \,(\textrm{resp.\,}\geq)\,\,  q\int_{\Om}u^{q-1}\phi^2 dx, 
	\end{aligned} \right.  
	\quad \forall\, \phi \in X_0 \setminus \{0\}.
	\end{equation}
\end{definition}

\begin{proposition} \label{propstable}
	Assume $p,q\in(1, N_s)$ and $\mu, \nu$ are positive functions in $L^r(\Gw)$ for some $r>\frac{N}{2s}$ such that $\| \mu \|_{L^r(\Omega)}= \| \nu \|_{L^r(\Omega)}=1$. For $\rho>0$ and $\tau>0$ small,  let $(\underline u_{\rho \mu},\underline v_{\tau \nu})$ be the minimal solution of \eqref{source-3} obtained in Theorem \ref{existcoup-1}. There exists $t_0 >0$ such that if $\max\{ \rho,\tau \}<t_0$ then $(\underline u_{\rho \mu},\underline v_{\tau \nu})$ is stable. Moreover, there exists a positive constant $C=C(N,s,p,q,t_0)$ such that
	\begin{equation} \label{strictstab}  \left\{ \begin{aligned} 
	\|\phi\|^2_{X_0}- p\int_{\Om}\underline v_{\tau \nu}^{p-1}\phi^2 dx\geq C\|\phi\|^2_{X_0},  \\
	\|\phi\|^2_{X_0}- q\int_{\Om}\underline u_{\rho \mu}^{q-1}\phi^2 dx\geq C\|\phi\|^2_{X_0},		
	\end{aligned} \right.
	\quad \forall\, \gf \in X_0 \setminus \{ 0 \}. 
	\end{equation}
\end{proposition}
\begin{proof}
	{\bf Step 1: }We show that there exists $t_0>0$  such that $(\underline u_{\rho\mu},\underline v_{\tau\nu})$ is stable provided $\max\{ \rho,\tau \}<t_0$.
	
	Indeed,  from Remark \ref{Linf}, %we see that $(\underline u_{\rho\mu},\underline v_{\tau\nu})$ satisfies \eqref{leub} with the constant $K=K(N,s,p,q,\Gw,\rho,\tau)$ such that $K \to 0$ as $(\rho,\tau) \to (0,0)$. Consequently, 
	it follows that for any $\phi\in X_0 \setminus \{0\}$, there exists $t_0>0$ small such that if $\max\{ \rho,\tau  \}<t_0$, there hold
	$$ \begin{aligned} \int_{\Gw}  \underline v_{\tau \nu}^{p-1}\phi^2\, dx\leq \|\underline v_{\tau \nu}\|_{L^\infty(\Om)}^{p-1} \int_{\Gw}\phi^2dx\leq \frac{1}{p}\|\phi\|^2_{X_0}, \\
	\int_{\Gw}  \underline u_{\rho \mu}^{q-1}\phi^2\, dx\leq \|\underline u_{\rho \mu}\|_{L^\infty(\Om)}^{q-1} \int_{\Gw}\phi^2dx\leq \frac{1}{q}\|\phi\|^2_{X_0}.
	\end{aligned} $$
	This completes Step 1.\medskip

	\noindent {\bf Step 2:}  We prove \eqref{strictstab}.  Assume $(\rho,\tau) \in (0,t_0) \times (0,t_0)$ and put $$ \rho'=\frac{\rho+t_0}{2}, \quad \tau'=\frac{\tau+t_0}{2}. $$ 
	Set 
	$$ \ga=\max\left \{  \big(\frac{\rho}{\rho'}\big)^\frac{1}{q}, \big(\frac{\tau}{\tau'}\big)^\frac{1}{p} \right\}<1. $$ 
	Let $(\unl u_{\rho' \mu},\unl v_{\tau' \nu})$ and $(\unl u_{\rho \mu},\unl v_{\tau \nu})$ be the solutions of \eqref{source-3} with data $(\rho' \mu, \tau' \nu)$ and $(\rho \mu, \tau \nu)$ respectively. Since $p,\,q>1$ and $\al<1$, it is easy to see that 
	$$ \begin{aligned}
	\ga \unl u_{\rho' \mu} &=\BBG_s[\ga \unl v_{\tau' \nu}^p] + \BBG_s[\ga \rho' \mu] \geq \BBG_s[(\ga \unl v_{\tau'\nu})^p] + \BBG_s[\rho \mu], \\
	\ga \unl v_{\tau' \nu} &=\BBG_s[\ga \unl u_{\rho' \mu}^q] + \BBG_s[\ga \tau' \nu ] \geq \BBG_s[(\ga \unl u_{\rho' \mu})^q] + \BBG_s[\tau \nu]. \\
	\end{aligned} $$
	Consequently, in view of the proof of Lemma \ref{compa}, we deduce $\ga \unl u_{\rho' \mu}\geq \unl u_{\rho \mu}$ and $\ga \unl v_{\tau' \nu}\geq \unl v_{\tau \nu}$. Furthermore, since $(\rho', \tau') \in (0,t_0) \times (0,t_0)$, by Step 1, we assert that $(\unl u_{\rho'\mu}, \unl v_{\tau' \nu})$ is stable. Therefore,
	\begin{equation} \label{stab2}  \begin{aligned} 
	0<\|\phi\|^2_{X_0}- p\int_{\Gw} \unl v_{\tau' \nu}^{p-1}\phi^2 dx  &\leq \|\phi\|^2_{X_0}- p\ga^{1-p}\int_{\Gw} \unl v_{\tau \nu}^{p-1}\phi^2 dx\\
	&=\al^{1-p}\big(\al^{p-1}\|\phi\|^2_{X_0}- p\int_{\Om}\underline v_{\tau \nu}^{p-1}\phi^2 dx\big).
	\end{aligned} \end{equation}
	Hence,
	\begin{equation} \begin{aligned}
	\|\phi\|^2_{X_0}- p\int_{\Gw} \unl v_{\tau \nu}^{p-1}\phi^2 dx &= (1-\al^{p-1})\|\phi\|^2_{X_0}+\al^{p-1}\|\phi\|^2_{X_0}- p\int_{\Gw} \unl v_{\tau \nu}^{p-1}\phi^2 dx\\
	&>(1-\ga^{p-1})\|\phi\|^2_{X_0}.
	\end{aligned} \end{equation}
	Similarly, one can prove
	$$ \|\phi\|^2_{X_0}- q\int_{\Gw} \unl u_{\rho \mu}^{q-1}\phi^2 dx >(1-\ga^{q-1})\|\phi\|^2_{X_0}.
	$$	
	Hence \eqref{strictstab} holds with $C=\min\{ 1-\ga^{p-1}, 1- \ga^{q-1} \}$.
\end{proof}

The norm of an element $z=(u,v)\in X_0\times X_0$ is defined by
$$\|z\|_{X_0\times X_0}:=\|(u,v)\|_{X_0\times X_0}=\big(\|u\|^2_{X_0}+\|v\|^2_{X_0}\big)^\frac{1}{2}.$$

\begin{definition}
	Let $(X, \|.\|_{X})$ be a real Banach space with its dual $(X^*, \|.\|_{X^*})$ and $I\in C^1(X,\R)$. For $c\in\R$, we say that $I$ satisfies Cerami condition at level $c$ (in short, $(C)_c$) if for any  sequence $\{w_n\}\subset X$ with
	$$I(w_n)\to c, \quad \|I'(w_n)\|_{X^*}(1+\|w_n\|_X)\to0,$$
	there is a subsequence $\{w_{n_k}\}$ of $\{w_n\}$ such that $\{w_{n_k}\}$ converges strongly in $X$.
	
	We say that $\{w_n\}\subset X$ is a Palais-Smale sequence of $I$ at level $c$ if $$I(w_n)\to c, \quad \|I'(w_n)\|_{X^*}\to0.$$
\end{definition}

The energy functional associated to \eqref{MP-prob} is
\begin{equation} \label{I} \begin{aligned}
 I(u, v)&:=\displaystyle\int_{\Rn\times\Rn}\frac{\big(u(x)-u(y)\big)\big(v(x)-v(y)\big)}{|x-y|^{N+2s}}dxdy-\int_{\Om} H( \underline v_{\tau \nu}, v) dx \\
&\qquad-\displaystyle\int_{\Om} \tilde H(\underline u_{\rho \mu}, u) dx \quad\forall\, (u, v)\in X_0\times X_0,
\end{aligned} \end{equation}
where 
\begin{equation}\begin{aligned}
&H(r,t):= \frac{1}{p+1}\bigg[(r+t^+)^{p+1}-r^{p+1}-(p+1)r^p t^+\bigg],\\
&\tilde H(r,t):= \frac{1}{q+1}\bigg[(r+t^+)^{q+1}-r^{q+1}-(q+1)r^q t^+\bigg], \quad r \geq 0.
\end{aligned}\end{equation}
Therefore,
\Bea
I'(u,v)(\phi,\psi)&=\displaystyle\int_{\Rn\times\Rn}\frac{\big(\phi(x)-\phi(y)\big)\big(v(x)-v(y)\big)}{|x-y|^{N+2s}}dxdy\\
&\quad+\displaystyle\int_{\Rn\times\Rn}\frac{\big(u(x)-u(y)\big)\big(\psi(x)-\psi(y)\big)}{|x-y|^{N+2s}}dxdy\\
&\quad-\displaystyle\int_{\Om} h(\underline v_{\tau \nu}, v)\psi dx-\int_{\Om} \tilde h(\underline u_{\rho \mu}, u)\phi dx,
\Eea
where
 $$h(r,t):=(r+t^+)^p-r^p \quad\text{and}\quad \tilde h(r,t):=(r+t^+)^q-r^q, \quad r \geq 0.$$
It is easy to see that if $z=(u,v)$ is a critical point of $I$ then $(u,v)$ solves \eqref{MP-prob}. We will find these critical points using Linking Theorem in the spirit of \cite{FMR}.

\begin{lemma}\lab{l:30-5-3} (i) There hold
	\be\lab{20-12-17-3} \qquad \frac{1}{p+1}t^{p+1} < H(r,t), \quad  \frac{1}{q+1}t^{q+1}< \tilde H(r,t) \quad\text{for}\quad r,\, t> 0.\ee
	
	(ii) Given any $M>0$, there exist $\theta>2$ and $T>0$ such that 
	\be\lab{20-12-17-2}\begin{aligned}
		H(r,t)\leq \frac{1}{\theta}h(r,t)t, \quad\tilde H(r,t)\leq \frac{1}{\theta}\tilde h(r,t)t, \quad\text{for}\quad 0\leq r\leq M,\,\, t\geq T,
	\end{aligned}\ee
	where $T$ depend on $M, p, q, \theta$.
	
	(iii) Let $0<\kappa<p+1$, then there exists a constant $C=C(p,q,\kappa)>0$ such that
	\be\lab{l:4-9-1}H(r, t), \, \tilde H(r,t)\geq t^\kappa-C  \quad\text{for}\quad r,\, t> 0. \ee
\end{lemma}
\begin{proof}
	 (i) Estimate \eqref{20-12-17-3} was proved in  \cite[Lemma C.2(ii)]{NS}. 
	 
	(ii) First let us choose $\theta_1\in (2, p+1)$ arbitrarily and fix it. Next, we define
	$$y(r, t):=h(r, t)t-\theta_1 H(r,t).$$
	From the definition of $h(r,t)$ and $H(r,t)$, a straight forward computation yields that
$$y(r,t)=t^2\bigg[p\big(1-\frac{\theta_1}{2}\big)r^{p-1}+\frac{p(p-1)}{2}\big(1-\frac{\theta_1}{3}\big)r^{p-2}t+\cdots+\big(1-\frac{\theta_1}{p+1}\big)t^{p-1}\bigg].$$	
Therefore, there exits $0<T= T(p, M, \theta_1)$ such that $y(r,t)>0$ for $t\geq T,\, r\leq M$. Similarly we can prove the other inequality by choosing $\theta_2\in(2, q+1)$.  Then by take $\theta=\min\{\theta_1,\theta_2\}$, we obtain \eqref{20-12-17-2}.
	
	(iii) Since $\kappa<p+1\leq q+1$, applying Young's inequality, we have
	$$t^\kappa \leq\frac{1}{p+1}t^{p+1}+c_1 \quad\text{and}\quad t^\kappa \leq \frac{1}{q+1}t^{q+1}+c_2$$
	where $c_1=c_1(\kappa,p)$ and $c_2=c_2(\kappa,q)$.
	Taking $C=\max\{c_1, c_2\}$, it follows $$\frac{1}{p+1}t^{p+1}, \, \frac{1}{q+1}t^{q+1}\geq t^\kappa-C. $$
	Combining this with (i), \eqref{l:4-9-1} follows.  
\end{proof}
\begin{remark}
	Combining  Lemma \ref{l:30-5-3} along with the fact that $H(r,t)=0$, for $t\leq 0$, it holds 
	\be\lab{30-5-4}
	H(r,t)\geq 0, \quad \tilde H(r,t)\geq 0, \quad \forall\, t\in\R , \, \forall r\geq 0.
	\ee
\end{remark}

We also observe that  (\cite[Lemma C.2(iii)]{NS}) for any $\eps>0$, there exists $c_{\eps}>0$, such that 
\be\lab{20-12-17-1} H(r,t)-\frac{p}{2}r^{p-1}t^2\leq \eps r^{p-1}t^2+c_{\eps}t^{p+1}, \quad r,\, t\geq 0.\ee

{\bf Notation}: For the rest of this section, we denote by $\| \cdot \|$, the norm in $X_0$, by $\|(\cdot,\cdot)\|$ the norm in $X_0\times X_0$ and by $\<\cdot,\cdot\>$ the inner product in $X_0$. \medskip

Next, we prove that $I$ has the geometry of the Linking theorem. 

\subsection {Geometry of the Linking Theorem}
We define,
$$E^+:=\{(u, u)\,:\, u\in X_0\}\quad\text{and}\quad E^-:=\{(u, -u)\,:\, u\in X_0\}.$$
\begin{lemma}\lab{l:30-5-1}
	There exist $\varrho,\, \si>0$ such that $I(u, v)\geq\si$ for all $(u,v)\in S:=\pa B_{\varrho}\cap E^+$.
\end{lemma}
\begin{proof}
	From the definition of $I(u,u)$, we have
	\Bea
	I(u, u)&=&\frac{1}{2}\bigg(\|u\|^2-p\Iom \underline v_{\tau \nu}^{p-1}u^2dx\bigg)+\frac{1}{2}\bigg(\|u\|^2-q\Iom \underline u_{\rho \mu}^{q-1}u^2dx\bigg)\\
	&&\qquad-\bigg(\Iom H(\underline v_{\tau \nu}, u)dx-\frac{p}{2}\Iom \underline v_{\tau \nu}^{p-1}u^2 dx\bigg) -\bigg(\Iom\tilde H(\underline u_{\rho \mu}, u)dx-\frac{q}{2}\Iom \underline u_{\rho \mu}^{q-1}u^2dx \bigg).
	\Eea
	Applying \eqref{20-12-17-1} and \eqref{strictstab}  to the above line and using \eqref{29-5-2} and Sobolev inequality, we obtain 
	\Bea
	I(u, u)&\geq& C\|u\|^2-\eps\Iom \underline v_{\tau \nu}^{p-1}u^2dx-C_\eps\Iom u^{p+1}dx-\eps\Iom \underline u_{\rho \mu}^{q-1}u^2dx-C_\eps\Iom u^{q+1}dx\\
	&\geq&(C-M^{p-1}S^{-1}\eps-M^{q-1}S^{-1}\eps)\|u\|^2-C\|u\|^{p+1}-C\|u\|^{q+1},
	\Eea
	where $S$ is the Sobolev constant. Now, choosing $\eps>0$ and $\varrho>0$ small enough, we find one $\sigma>0$ such that 
	$I(u, u)\geq \si$ when $\|u\|=\varrho$, as $p, \, q>1$. This proves the lemma.
\end{proof}

Let $\psi_0\in X_0$ be a fixed nonnegative function with $\|\psi_0\|=1$ and
$$Q_{\psi_0}:=\{r(\psi_0, \psi_0)+w\,:\, w\in E^-, \|w\|\leq R_0,\, 0\leq r\leq R_1 \}.$$

\begin{lemma}\lab{l:30-5-2}
	There exist constants $R_0,\, R_1>0$, which depend on $\psi_0$, such that
	$I(u, v)\leq 0$ for all $(u,v)\in \pa Q_{\psi_0}$.
\end{lemma}
\begin{proof}
	We note that boundary $\pa Q_{\psi_0}$ of the set $Q_{\psi_0}$ is taken in the space $\R(\psi_0, \psi_0)\oplus E^-$ and consists of three parts. We estimate $I$ on these parts as below.
	
	\vspace{2mm}
	
	\noi {\bf Case 1}: $z\in\pa Q_{\psi_0}\cap E^-$ and of the form $z=(u,-u)\in E^-$. Then, thanks to \eqref{30-5-4}, it follows
	$$I(z)=-\|u\|^2-\Iom H(\underline v_{\tau \nu}, -u)dx-\Iom\tilde H(\underline u_{\rho \mu}, u)dx\leq 0.$$
	
	\noi {\bf Case 2}: $z=R_1(\psi_0, \psi_0)+(u,-u)\in \pa Q_{\psi_0}$ with $\|(u, -u)\|\leq R_0$. Thus,	
	\be\lab{4-9}I(z)=R_1^2\|\psi_0\|^2-\|u\|^2-\Iom H(\underline v_{\tau\nu}, R_1\psi_0-u)dx-\Iom\tilde H(\underline u_{\rho\mu},  R_1\psi_0+u)dx.\ee
	
	For $2<\kappa<p+1 \leq q+1$, set 
	\begin{equation*} \xi(t): =\left\{ \begin{aligned} 
t^\kappa \quad &&\text{if }\, t\geq 0, \\
0 \quad &&\text{if }\, t<0. 
\end{aligned} 
\right. \end{equation*}
	Then, applying \eqref{30-5-4} and \eqref{l:4-9-1} to \eqref{4-9}, we get
	\Bea I(z)&\leq& R_1^2-\Iom \xi(R_1\psi_0-u)dx-\Iom \xi(R_1\psi_0+u)dx+C
	\Eea
where $C=C(p,q,\kappa)$ is the constant in \eqref{l:4-9-1}.	
	%&\leq& R_1^2-\frac{R_1^{p+1}}{p+1}\Iom |\psi_0-\frac{u}{R_1}|^{p+1}dx-\frac{R_1^{q+1}}{q+1}\Iom |\psi_0+\frac{u}{R_1}|^{q+1}dx.
	%Set, $$\xi_1(t):=\frac{1}{p+1}|t|^{p+1}, \quad \xi_2(t):=\frac{1}{q+1}|t|^{q+1}.$$
%Using convexity property of $\xi_1$ and $\xi_2$, we have
%$$\xi_2(a+b)+\xi_1(a-b)\geq \xi_2(a)+\xi_1(a)+\xi_2'(a)b-\xi_1'(a)b.$$
%Taking $a=R_1\psi_0$ and $b=u$, we get

Now, using convexity of the function $\xi$, we obtain
	$$ I(z)\leq R_1^2-2R_1^\kappa\Iom |\psi_0|^\kappa dx + C.$$ 
	Therefore, since $\kappa>2$, taking $R_1$ large enough (depending on $\psi_0$), it follows that $I(z)\leq 0$.
		
		\iffalse+++++++++++++++
	Using H\"{o}lder inequality followed by Sobolev inequality along with the fact that $\|\psi\|=1$ and $\|u\|\leq \frac{R_0}{\sqrt{2}}$ we have
	\Bea I(z)&\leq& R_1^2-\frac{R_1^{p+1}}{p+1}\|\psi_0\|_{L^{p+1}(\Om)}^{p+1}- \frac{R_1^{q+1}}{q+1}\|\psi_0\|_{L^{q+1}(\Om)}^{q+1}+ \frac{1}{\sqrt{2}}R_1^pR_0 S^{-p}.
	\Eea
	
	Now to make $I(z)\leq 0$, we choose $R_0$ and $R_1$ such that 
	\be\lab{4-9-2}R_0\leq R_1S^p\sqrt{2}\bigg[\frac{1}{p+1}\|\psi_0\|_{L^{p+1}(\Om)}^{p+1}+\frac{1}{q+1}R_1^{q-p}\|\psi_0\|_{L^{q+1}(\Om)}^{q+1}-R_1^{1-p}\bigg].\ee
	++++++++++++\fi
	
	\vspace{2mm}
	
	\noi {\bf Case 3}: $z=r(\psi_0, \psi_0)+(u,-u)\in \pa Q_{\psi_0}$ with $\|(u, -u)\|= R_0$ and $0\leq r\leq R_1$.
	
	Then, using \eqref{30-5-4} it follows that
	$$I(z)\leq r^2\|\psi_0\|^2-\|u\|^2\leq R_1^2-\frac{1}{2}R_0^2.$$ Choosing $R_0\geq \sqrt{2}R_1$, we have $I(z)\leq 0$.
	
	Combining case 2 and case 3, in order that the geometry of Linking theorem holds, we choose $R_0$, $R_1$ large enough with $R_0\geq \sqrt{2}R_1$.
		\end{proof}

Our next aim is to prove that Cerami sequences are bounded.

\begin{proposition}\lab{p:30-5-1}
	Let $(u_m, v_m)\in X_0\times X_0$ such that
	
	(i) $I(u_m, v_m)=c+\de_m$, where $\de_m\to 0$ as $m\to\infty$.
	
	(ii) $(1+\|(u_m, v_m)\|)|I'(u_m, v_m)(\phi,\psi)|\leq \eps_m \|(\phi, \psi)\|$ for $\phi,\psi\in X_0\times X_0$ and $\eps_m\to 0$ as $m\to\infty$.
	Then,
	\begin{eqnarray*}
		\|u_m\|\leq C, &\quad  \|v_m\|\leq C\\
		\Iom h(\underline v_{\tau \nu}, v_m)v_mdx\leq C, &\quad \displaystyle\Iom \tilde h(\underline u_{\rho \mu}, u_m)u_mdx\leq C\\
		\Iom H(\underline v_{\tau \nu}, v_m)dx\leq C, &\quad \displaystyle\Iom \tilde H(\underline u_{\rho \mu}, u_m)dx\leq C.
	\end{eqnarray*}
\end{proposition}
\begin{proof}
	Choosing $(\phi,\psi)=(v_m,0)$ and $(\phi,\psi)=(0, u_m)$ in (ii), we have
	\begin{equation} \label{21-2-2} \begin{aligned}
	\left| \| v_m \|^2 - \int_{\Omega} \tilde h(\underline u_{\rho\mu},u_m)v_m dx\right| \leq \varepsilon_m \| v_m \|, \\
	\left| \| u_m \|^2 - \int_{\Omega} h(\underline v_{\tau\nu},v_m)u_mdx \right| \leq \varepsilon_m \| u_m \|. 
	\end{aligned} \end{equation}
	Now choosing $(\phi,\psi)=(u_m,0)$ and $(\phi,\psi)=(0, v_m)$ in (ii), we have
	\begin{equation} \label{21-2-2b} \begin{aligned}
	\left|\<u_m,v_m\>-\int_{\Omega} \tilde h(\underline u_{\rho \mu},u_m)u_m dx\right|\leq \varepsilon_m,\\
	\bigg|\<u_m,v_m\>-\Iom h(\underline v_{\tau \nu},v_m)v_m dx\bigg|\leq\varepsilon_m.
	\end{aligned} \end{equation}
	On the other hand, from (i), we obtain 
	\begin{equation}\label{21-2-3}
	\<u_m,v_m\>-\Iom H(\underline v_{\tau \nu}, v_m)dx-\Iom \tilde H(\underline u_{\rho \mu}, u_m)dx=c+\de_m.
	\end{equation}
	Combining \eqref{21-2-2} and \eqref{21-2-3} and using \eqref{20-12-17-2}, we get
	\Bea
	2c+2\de_m&=&\<u_m,v_m\>-2\Iom H(\underline v_{\tau \nu}, v_m)dx+\<u_m,v_m\>-2\Iom \tilde H(\underline u_{\rho \mu}, u_m)dx\\
	&\geq&-2\eps_m+\Iom \tilde h(\underline u_{\rho \mu},u_m)u_mdx+\Iom h(\underline v_{\tau \nu},v_m)v_mdx \\
	&&-2\Iom \tilde H(\underline u_{\rho \mu}, u_m)dx-2\Iom H(\underline v_{\tau \nu}, v_m)dx\\
	&\geq& -2\eps_m+(\theta-2)\int_{\Om\cap\{v_m>T\}} H(\underline v_{\tau \nu},v_m)dx+(\theta-2)\int_{\Om\cap\{u_m>T\}} \tilde H(\underline u_{\rho \mu},u_m)dx+ C,
	\Eea
	where we have used the fact %$\int_{\Om\cap\{v_m\leq T\}}h(\underline v_{\tau \nu}, v_m)v_m<C$ and 
	$$\int_{\Om\cap\{v_m\leq T\}}H(\underline v_{\tau\nu}, v_m)dx<C \quad \text{and} \quad \int_{\Om\cap\{u_m\leq T\}}\tilde H(\underline u_{\rho\mu}, v_m)dx<C,
	$$ 
	which follows from the definition of $h(r,t)$, $H(r,t)$ and \eqref{leub}. Therefore, 
	\begin{equation}\label{29-5-1}
	\int_{\Omega} H(\underline v_{\tau \nu},v_m)dx\leq C \quad\text{and}\quad \int_{\Omega}\tilde H(\underline u_{\rho \mu},u_m)dx\leq C.
	\end{equation}
	Using \eqref{20-12-17-2}, similarly it can be also shown that 
	\be\lab{30-5-8}
	\Iom h(\underline v_{\tau \nu}, v_m)v_mdx\leq C,\quad \Iom \tilde h(\underline u_{\rho \mu}, u_m)u_mdx\leq C.
	\ee
	Observe that $h(\underline v_{\tau \nu}, v_m)=0$ if $v_m\leq 0$ and $h(\underline v_{\tau \nu},v_m)u_m\leq 0$ if $u_m\leq 0$. Therefore applying Young's inequality, \eqref{20-12-17-3}, \eqref{29-5-1} and the fact that $\underline u_{\rho \mu}$ and $\underline v_{\tau \nu}$ are bounded (see \eqref{29-5-2}) yields
	\begin{equation} \label{31-5-1} \begin{aligned}
	\Iom h(\underline v_{\tau \nu},v_m)&u_mdx \leq \int_{\Om\cap\{v_m\geq 0,\, u_m\geq 0\}}h(\underline v_{\tau \nu},v_m)u_mdx\\
	&\leq\frac{1}{q+1}\int_{\Om\cap\{v_m\geq 0,\, u_m\geq 0\}}u_m^{q+1}dx+\frac{q}{q+1}\int_{\Om\cap\{v_m\geq 0,\, u_m\geq 0\}}h(\underline v_{\tau \nu},v_m)^\frac{q+1}{q}dx\\
	&\leq\Iom \tilde H(\underline u_{\rho \mu}, u_m)dx+ \frac{q}{q+1}\int_{\Om\cap\{v_m\geq 0,\, u_m\geq 0\}} (\underline v_{\tau \nu}+v_m^+)^\frac{p(q+1)}{q}dx\\
	&\leq C_1+C(q)\bigg(\Iom \underline v_{\tau \nu}^\frac{p(q+1)}{q}dx+ \int_{\Om\cap\{v_m\geq 0,\, u_m\geq 0\}} v_m^\frac{p(q+1)}{q}dx\bigg)\\
	&\leq C_1+ C_2+ C_3\bigg(\int_{\Om\cap\{v_m\geq 0,\, u_m\geq 0\}} v_m^{p+1}dx\bigg)^\frac{(q+1)p}{(p+1)q}|\Om|^\frac{q-p}{(p+1)q}\\
	&\leq C_1+ C_2+ C_4 \bigg(\int_{\Om\cap\{v_m\geq 0,\, u_m\geq 0\}}H(\underline v_{\tau \nu}, v_m)dx\bigg)^\frac{(q+1)p}{(p+1)q} <C.
	\end{aligned} \end{equation}
	In above estimate we have also used the fact $p \leq q$ implies $(q+1)p/q \leq p+1$. 
	
	Similarly we can show that $\Iom\tilde h(\underline u_{\rho \mu},u_m)v_mdx<C$. Therefore substituting back in \eqref{21-2-2}, we obtain $\|u_m\|\leq C$ and $\|v_m\|\leq C$. 
\end{proof}

\subsection{Finite dimensional problem} Since the functional $I$ is strongly indefinite and defined in infinite dimensional space, no suitable linking theorem is available. We therefore approximate \eqref{MP-prob} with a sequence of finite dimensional problems.

Associated to the eigenvalues $0<\la_1<\la_2\leq\la_3\leq\cdots\to\infty$ of $((-\De)^s, X_0)$, there exits an orthogonal basis $\{\va_1, \va_2, \cdots\}$  
of corresponding eigen functions in $X_0$ and $\{\va_1, \va_2, \cdots\}$ is an orthonormal basis for $L^2(\Om)$. We set
\begin{eqnarray*}
	E_n^+ &:=&\text{span}\{(\va_i, \va_i)\,:\, i=1,2, \cdots, n\},\\
	E_n^-&:=&\text{span}\{(\va_i, -\va_i)\,:\, i=1,2, \cdots, n\},\\
	E_n&:=&E_n^+\oplus  E_n^-.
\end{eqnarray*}
Let $\psi_0\in X_0$ be a fixed nonnegative function with $\|\psi_0\|=1$ and
$$Q_{n,\psi_0}:=\{r(\psi_0, \psi_0)+w\,:\, w\in E_n^-, \|w\|\leq R_0,\, 0\leq r\leq R_1 \},$$
where $R_0$ and $R_1$ are chosen in Lemma \ref{l:30-5-2}. Here we recall that these constants depend only on $\psi_0,\, p,\, q$. Next, define
$$H_{n,\psi_0}:=\R(\psi_0,\psi_0)\oplus E_n,\quad H^+_{n,\psi_0}:=\R(\psi_0,\psi_0)\oplus E_n^+,\quad H^-_{n,\psi_0}:=\R(\psi_0,\psi_0)\oplus E_n^-.$$
$$\Ga_{n,\psi_0}:=\{\pi\in C(Q_{n,\psi_0}, H_{n,\psi_0})\,:\, \pi(u, v)=(u,v)\,\, \text{on}\,\, \pa Q_{n,\psi_0}\},$$ and 
$$c_{n,\psi_0}:=\inf_{\pi\in \Ga_{n,\psi_0}}\max_{(u,v)\in Q_{n,\psi_0}}I\big(\pi(u,v)\big).$$
Using an intersection theorem (see \cite[Proposition 5.9]{R}), we have
$$\pi(Q_{n,\psi_0})\cap (\pa B_{\varrho}\cap E^+)\not=\emptyset, \quad\forall\, \pi\in\Ga_{n,\psi_0}.$$
Thus there exists  an $(u,v)\in (Q_{n,\psi_0})$ such that $\pi(u,v)\in \pa B_{\varrho}\cap E^+$. Combining this with Lemma \ref{l:30-5-1}, we get $I(\pi(u,v))\geq\si$. This in turn implies $c_{n,\psi_0}\geq \si>0$. Our next goal is to show that $c_{n,\psi_0}$ has an upper bound. For that, we observe that the identity map $\text{Id} : Q_{n,\psi_0}\to H_{n,\psi_0}$ is in $\Ga_{n,\psi_0}$. Thus for an element of the form $z:=r(\psi_0,\psi_0)+(u,-u)\in Q_{n,\psi_0}$, we compute
\Bea
I(z)&=&\<r\psi_0+u, r\psi_0-u\>-\Iom H(\underline v_{\tau\nu}, r\psi_0-u)dx-\Iom\tilde H(\underline u_{\rho\mu}, r\psi_0+u)dx\\
&=&r^2-\|u\|^2-\bigg[\Iom H(\underline v_{\tau\nu}, r\psi_0-u)dx+\Iom\tilde H(\underline u_{\rho\mu}, r\psi_0+u)dx\bigg]\leq R_1^2,\\
\Eea
where in the last inequality we have used \eqref{30-5-4}. Consequently,
$$ \max_{z \in Q_{n,\psi_0}}I(z) \leq R_1^2.
$$
Therefore,
$$ c_{n,\psi_0} \leq \max_{z \in Q_{n,\psi_0}}I(\text{Id}(z)) = \max_{z \in Q_{n,\psi_0}}I(z) \leq R_1^2.
$$
Hence $0<\si\leq c_{c,\psi_0}\leq R_1^2$. We remark here that upper and lower bound do not depend on $n$. Define,
$$I_{n,\psi_0}:= I\big|_{H_{n,\psi_0}}.$$
Thus, in view of Lemmas \ref{l:30-5-1} and \ref{l:30-5-2}, we see that geometry of Linking theorem holds for the functional $I_{n,\psi_0}$. Hence applying the linking theorem \cite[Theorem 5.3]{R} to $I_{n,\psi_0}$, we obtain a Palais-Smale sequence, which is bounded in view of Proposition \ref{p:30-5-1} (also see \cite[pg. 1046]{FMR}). Therefore, using the fact that $H_{n,\psi_0}$ is a finite dimensional space, we obtain the following proposition:

\begin{proposition}\lab{p:30-5-2}
	For every $n\in\N$ and for every  $\psi_0\in X_0$, a fixed nonnegative function with $\|\psi_0\|=1$, the functional $I_{n,\psi_0}$ has a critical point $z_{n,\psi_0}$ such that
	\be\lab{30-5-5}z_{n,\psi_0}\in H_{n,\psi_0}, \quad I_{n,\psi_0}'(z_{n,\psi_0})=0,\quad I_{n,\psi_0}(z_{n,\psi_0})=c_{n,\psi_0}\in[\si, R_1^2],\ee
	\be\lab{30-5-6}  \quad \|z_{n,\psi_0}\|\leq C,\ee
	where $C$ does not depend on $n$.
\end{proposition}

\subsection{Existence of solution of \eqref{MP-prob}.}

{\bf Step 1:} Let $\psi_0\in X_0$ be a fixed nonnegative function with $\|\psi_0\|=1$. Then applying Proposition \ref{p:30-5-2}, we get a sequence $\{z_{n,\psi_0}\}_{n=1}^\infty$ satisfying \eqref{30-5-5} and \eqref{30-5-6}. Consequently, there exists $(u_0, v_0)\in X_0\times X_0$ such that
\be\lab{30-5-7}z_{n,\psi_0}:=(u_{n,\psi_0}, v_{n,\psi_0})\rightharpoonup (u_0, v_0) \quad\text{in}\quad X_0\times X_0, 
\ee
\be\lab{30-5-7'}u_{n,\psi_0}\to u_0,\quad v_{n,\psi_0}\to v_0 \quad\text{in}\quad L^r(\Rn), \,1\leq r<2^*_s \quad\text{and a.e. in}\quad \Om.   
\ee
Further, applying Proposition \ref{p:30-5-1}, we conclude
\be\lab{30-5-9}
\Iom h(\underline v_{\tau \nu}, v_{n,\psi_0})v_{n,\psi_0}dx\leq C, \quad \displaystyle\Iom \tilde h(\underline u_{\rho \mu}, u_{n,\psi_0})u_{n,\psi_0}dx\leq C.
\ee
\be\lab{30-5-10}
\Iom H(\underline v_{\tau \nu}, v_{n,\psi_0})dx\leq C, \quad \displaystyle\Iom \tilde H(\underline u_{\rho \mu}, u_{n,\psi_0})dx\leq C.
\ee
Next, taking as test functions $(0,\psi)$ and $(\phi, 0)$ in \eqref{30-5-5}, where $\phi,\,\psi$ are arbitrary functions in $F_n:=\text{span}\{\va_i\,:\, i=1,2,\cdots, n\}$, we obtain
\be\lab{30-5-11}
\<\psi, u_{n,\psi_0}\>=\Iom h(\underline v_{\tau \nu}, v_{n,\psi_0})\psi dx\quad\forall\, \psi\in F_n,
\ee
\be\lab{30-5-12}
\<\phi, v_{n,\psi_0}\>=\Iom \tilde h(\underline u_{\rho \mu}, u_{n,\psi_0})\phi dx \quad\forall\, \phi\in F_n.
\ee
Now applying \eqref{29-5-2}, \eqref{30-5-7'} and the fact that $p<N_s<2^*_s$, we also have $h(\underline v_{\tau \nu}, v_{n,\psi_0})$ and $h(\underline v_{\tau \nu}, v_0)$ are $L^1$ functions. Therefore, using \eqref{30-5-9}, \eqref{30-5-7'} and an argument similar to the one used in \cite[Lemma 2.1]{FMR1}, it follows that 
$$h(\underline v_{\tau \nu}, v_{n,\psi_0})\to h(\underline v_{\tau \nu}, v_0), \quad\tilde h(\underline u_{\rho \mu}, u_{n,\psi_0})\to \tilde h(\underline u_{\rho \mu}, u_0)    \quad\text{in}\quad L^1(\Om).$$ 
Hence, taking the limit in \eqref{30-5-11} and \eqref{30-5-12} and using the fact that $\cup_{n=1}^\infty F_n$ is dense in $X_0$, it follows that
\be
\<\psi, u_{0}\>=\Iom h(\underline v_{\tau \nu}, v_0)\psi dx \quad \text{and}\quad \<\phi, v_0\>=\Iom \tilde h(\underline u_{\rho \mu}, u_0)\phi dx,\no
\ee
for all $\phi, \psi\in X_0$. As a consequence, 
\be\lab{31-5-3}(-\De)^s u_0=h(\underline v_{\tau \nu}, v_0) \quad \text{and}\quad (-\De)^s v_0=\tilde h(\underline u_{\rho \mu}, u_0),\quad u_0=v_0=0 \quad\text{in}\, \Om^c.\ee

\vspace{3mm}

{\bf Step 2:} In this step we show that $u_0$ and $v_0$ are nontrivial and nonnegative.

\vspace{2mm}

Suppose not, we assume $u_0\equiv 0$ in $X_0$. Plugging back to the equation \eqref{31-5-3}, it implies $(-\De)^s v_0=0$. As a consequence
$$0=\Iom v_0(-\De)^s v_0\, dx =\Iom |(-\De)^\frac{s}{2}v_0|^2\, dx= \|v_0\|_{X_0}^2.$$
Therefore, $v_0\equiv 0$ in $X_0$, that is, $u_{n, \psi_0}\rightharpoonup 0$ and $v_{n, \psi_0}\rightharpoonup 0$ in $X_0$. Consequently, $u_{n, \psi_0}\to 0$ and $v_{n, \psi_0}\to 0$ in $L^r(\Om)$ for $1\leq r<2^*_s$. Since $p+1,\, q+1<2^*_s$, computing as in \eqref{31-5-1} we obtain
\Bea
\lim_{n\to\infty}\Iom h(\underline v_{\tau \nu},v_{n,\psi_0})u_{n,\psi_0}dx &\leq&\lim_{n\to\infty}C\bigg(\int_{\Om\cap\{v_{n,\psi_0}\geq 0,\, u_{n,\psi_0}\geq 0\}}u_{n,\psi_0}^{q+1}dx\\
&&\qquad\qquad+\int_{\Om\cap\{v_{n,\psi_0}\geq 0,\, u_{n,\psi_0}\geq 0\}}h(\underline v_{\tau \nu},v_{n,\psi_0})^\frac{q+1}{q}dx\bigg)\\
&\leq&\lim_{n\to\infty} C \int_{\Om\cap\{v_{n,\psi_0}\geq 0\}}[(\underline v_{\tau \nu}+v_{n,\psi_0})^p-\underline v_{\tau \nu}^p]^\frac{q+1}{q}dx\\
&\leq&\lim_{n\to\infty} C\Iom \bigg[|v_{n,\psi_0}|^p+|\underline v_{\tau\nu}|^{p-1}|v_{n,\psi_0}| \bigg]^\frac{q+1}{q}dx\\
&\leq&\lim_{n\to\infty} C\bigg[\Iom |v_{n,\psi_0}|^{\frac{p}{q}(q+1)}dx+\Iom|\underline v_{\tau \nu}|^{\frac{1}{q}(p-1)(q+1)}|v_{n,\psi_0}|^\frac{q+1}{q}dx\bigg]\\
&=&0,
\Eea
where for the last inequality we have used the fact that $\frac{p}{q}(q+1) \leq p+1$ (since $p \leq q$), \eqref{29-5-2} and the fact that $\frac{q+1}{q}<2^*_s$ (since $\frac{N-2s}{N+2s}<1<q$). Similarly it follows that $\Iom \tilde h(\underline u_{\rho \mu},u_{n,\psi_0})v_{n,\psi_0}dx \to 0$. As a consequence, taking $\psi=u_{n,\psi_0}$
in \eqref{30-5-11} and $\phi=v_{n,\psi_0}$ in \eqref{30-5-12} yields $\|u_{n,\psi_0}\|\to 0$ and $\|v_{n,\psi_0}\|\to 0$ respectively. Hence, $u_{n, \psi_0}\to 0$ and $v_{n, \psi_0}\to 0$ strongly in $X_0$. This in turn, implies
$\<u_{n,\psi_0}, v_{n,\psi_0}\>\to (0, 0)$. 

Further, combining  \eqref{20-12-17-1} along with \eqref{30-5-4} and \eqref{29-5-2} yields $\Iom H(\underline v_{\tau \nu}, v_{n,\psi_0})dx\to 0$ and $\Iom\tilde H(\underline u_{\rho \mu}, u_{n,\psi_0})dx\to 0$. Hence, $c_{n,\psi_0}=I_{n,\psi_0}(u_{n,\psi_0}, v_{n,\psi_0})\to 0$. This is a contradiction to the fact that $c_{n,\psi_0}\in [\si, R_1^2]$. Therefore, $u_0, \, v_0$ are nontrivial.

Since $u_0\in X_0$, by direct computation it is easy to see that $u_0^+,\, u_0^-\in X_0$.  Thus, taking the test function as $u_0^-$ for the first equation in \eqref{31-5-3}  yields $0\leq \<u_0^+, u_0^-\> - \|u_0^-\|^2\leq -\|u_0^-\|^2$, i.e., $u_0\geq 0$ a.e.. Similarly, $v_0\geq 0$. As a result, step 2 follows.

Hence we obtain the existence of a nonnegative nontrivial solution $(u_0, v_0)\in X_0\times X_0$ of \eqref{MP-prob}. 

\subsection{Proof of Theorem \ref{2nd sol} completed} 	
	In order to construct the second solution of \eqref{source-3}, we define
	\be\lab{31-5-4}u_{\rho \mu}=\underline u_{\rho \mu}+u_0, \quad v_{\tau \nu}=\underline v_{\tau \nu}+v_0,\ee
	where $(u_0, v_0)$ are as defined in Step 1 of Section 5.3. 
	Clearly $u_{\rho \mu}\geq \underline u_{\rho \mu}$ and $v_{\tau \nu}\geq \underline v_{\tau \nu}$. Moreover, as $u$ and $v$ are nontrivial element in $X_0$, there exist two positive measure sets $\Om', \Om''\subset\Om$ such that $u_0>0$ in $\Om'$ and , $v_0>0$ in $\Om''$. Thus $u_{\rho \mu}>\underline u_{\rho \mu}$ in $\Om'$ and $v_{\tau \nu}>\underline v_{\tau \nu}$ in $\Om''$. Further, as $(u_0, v_0)\in X_0\times X_0$ is a solution of \eqref{MP-prob}, we have
	\begin{equation} \label{ru} 
	\<u_0, \psi\>=\Iom h(\underline v_{\tau \nu}, v_0)\psi dx, \quad \<v_0, \phi\>=\Iom\tilde h(\underline u_{\rho \mu}, u_0)\phi dx, \quad \forall\, \phi, \, \psi \in X_0.
	\end{equation}
	Set 
	$$ \CT(\Gw):=\{ \tilde\psi \in C^\infty(\Gw): \text{ there exists } \psi \in C_0^\infty(\Gw) \text{ such that } \tilde\psi=\BBG_s[\psi]  \}.  $$
	This is a space of test function defined in \cite[Page 41]{A}. By \cite[Lemma 5.6]{A}, $\CT(\Gw) \subset X_0$. Therefore, we deduce from \eqref{ru} that
	\begin{equation}\begin{aligned} \label{ru2} \int_{\Gw} u_0 \fw \psi \,dx &= \<u_0, \psi\> = \Iom h(\underline v_{\tau \nu}, v_0)\psi dx\quad \forall\, \psi \in \CT(\Gw),\\
	\int_{\Gw} v_0 \fw \phi \,dx &= \<v_0, \phi\> = \Iom \tilde h(\underline u_{\rho \mu}, u_0)\phi dx\quad \forall\, \phi \in \CT(\Gw).
	\end{aligned}\end{equation}
	Moreover, \cite[Lemma 5.12 and Lemma 5.13]{A} ensures that $\CT(\Gw) \subset \BBX_s(\Gw)$ and 
	\begin{equation}\begin{aligned} \label{ru3} \int_{\Gw} u_0 \fw \psi \,dx &= \Iom h(\underline v_{\tau \nu}, v_0)\psi dx\quad \forall\, \psi \in \BBX_s(\Gw),\\
	\int_{\Gw} v_0 \fw \phi \,dx &= \Iom\tilde h(\underline u_{\rho \mu}, u_0)\phi dx\quad \forall\, \phi \in \BBX_s(\Gw).
	\end{aligned}\end{equation}
	This means that $(u_0, v_0)$ is a weak solution of 
	\begin{equation} \label{ru3} \left\{  
	\begin{aligned}
	\fw u_0 &=(\underline v_{\tau \nu}+v_0)^p- \underline v_{\tau \nu}^p \quad \text{in }\Gw\\
	\fw v_0 &=(\underline u_{\rho \mu}+u_0)^q- \underline u_{\rho \mu}^q \quad \text{in }\Gw\\
	u_0=v_0 &= 0 \quad \text{in } \Gw^c,
	\end{aligned} \right.
	\end{equation}
	Hence,  $(u_{\rho \mu}, v_{\tau \nu})$, as defined in \eqref{31-5-4}, is clearly a weak solution of \eqref{source-3}. 

	If $\mu,\nu \in L^r(\Gw) \cap L_{loc}^\infty(\Gw)$ then by Theorem \ref{reg},  $u_{\rho \mu}, v_{\tau \nu}$, $\underline u_{\rho \mu}, \underline v_{\tau \nu}\in C^{\alpha}_{loc}(\Om)$, for some $\al\in(0,2s)$.  Therefore,  $u_0, v_0\in C^{\alpha}_{loc}(\Om)$. Also since we have $(-\De)^s u_0$, $(-\De)^s v_0\geq 0$ in $\Om$ and $0\not\equiv u_0, v_0\geq 0$ in $\Rn$, applying the strong maximum principle \cite[Proposition 2.17]{Si}, we have $u_0, v_0>0$ in $\Omega$. Hence from \eqref{31-5-4}, we deduce $u_{\rho \mu}>\underline u_{\rho \mu}$ and $v_{\tau \nu}>\underline v_{\tau \nu}$. In view of Remark \ref{l:3-9-1}, this completes the proof. \qed

\vspace{3mm}

{\bf Acknowledgement:} M. Bhakta is partially supported by  DST/INSPIRE 04/2013/000152 and MATRICS grant MTR/2017/000168.


\vspace{3mm}



\begin{thebibliography}{XX}
	\bibitem{A}{\sc Abatangelo, N.,} Large solutions for fractional Laplacian operators, Ph.D. dissertation (2015).
	
	\bibitem{BN}{\sc Bhakta, M., Nguyen, P-T}; Boundary value problems with measures for fractional elliptic equations involving source nonlinearities. arXiv:1801.01544.
	
	\bibitem{BVi}{\sc  Bidaut-V\'{e}ron, M-F.; Vivier, L.}; An elliptic semilinear equation with source term involving boundary measures: the subcritical case. \textit{Rev. Mat. Iberoamericana} 16 (2000), no. 3, 477--513.
	
	\bibitem{BVY}{\sc  Bidaut-V\'{e}ron, M-F.; Yarur, C.},  Semilinear elliptic equations and systems with measure data: existence and a priori estimates. \textit{Adv. Differential Equations} 7 (2002), no. 3, 257--296.
	
	%\bibitem{BB}{\sc  Bogdan, K.,  Byczkowski, T.}; Potential theory for the $\alpha-$stable Schr�dinger operator on bounded Lipschitz domains. \textit{Studia Math.} 133 (1999), no. 1, 53--92. 
	
	\bibitem{BBC}{\sc  Benilan, P.; Brezis, H.; Crandall, M. G.}, A semilinear equation in $L^1(\Rn)$, \textit{Ann. Scuola Norm. Sup. Pisa Cl. Sci.} (4) 2 (1975), no. 4, 523--555. 
	
	\bibitem{CS1}{\sc Caffarelli, L.;  Silvestre, L}, The Evans--Krylov theorem for nonlocal fully non linear equations, \textit{Ann. Math.} 174(2) (2011) 1163--1187.
	
	\bibitem{CFV} {\sc Chen, H.; Felmer, P.; V\'eron, L.}, Elliptic equations involving general subcritical source nonlinearity and measures. arxiv:1409.3067 (2014).


	\bibitem{CQ}{\sc Chen, H.; Quaas, A.}, Classification of isolated singularities of nonnegative solutions to fractional semilinear elliptic equations and the existence results, \textit{J. London Math. Soc.} (2) 97 (2018) 196--221.
	
	
	%\bibitem{ChSo}{\sc Chen, Z.;  Song, R.}, Martin boundary and integral representation for harmonic functions of symmetric stable processes, \textit{J. Funct. Anal.} 159 (1998), 267--294.
	
	\bibitem{ChSo1}{\sc Chen, Z.;  Song, R.}, Estimates on Green functions and Poisson kernels for symmetric stable process, \textit{Math.
		Ann.} 312 (1998), 465--501.
	
	\bibitem{CV}{\sc Chen, H; V\'{e}ron, L},  Semilinear fractional elliptic equations involving measures. \textit{J. Differential Equations} 257 (2014), no. 5, 1457--1486.
	
	\bibitem{Co}{\sc  Cowan, C.} Liouville theorems for stable Lane-Emden systems with biharmonic problems. \textit{Nonlinearity} 26 (2013), no. 8, 2357--2371.
	
	\bibitem{DKK}{\sc Dahmani, Z.; Karami, F.; Kerbal, S.}, Nonexistence of positive solutions to nonlinear nonlocal elliptic
	systems. \textit{J. Math. Anal. Appl.} 346(1), 22--29 (2008).
	
	\bibitem{FF}{\sc de Figueiredo, D. G.; Felmer, P}, A Liouville-type theorem for elliptic systems,
	\textit{Ann. Scuola Norm. Sup. Pisa Cl. Sci.} (4) 21 (1994), 387--397.
	
	\bibitem{FMR1}{\sc  de Figueiredo, D. G.; Miyagaki, O. H.; Ruf, B.}, Elliptic equations in $\R^2$ with nonlinearities in the critical growth range. \textit{Calc. Var. Partial Differential Equations} 3 (1995), no. 2, 139--153.
	
	\bibitem{FMR}{\sc  de Figueiredo, D. G.; do , J. Marcos; Ruf, B.} Critical and subcritical elliptic systems in dimension two. \textit{Indiana Univ. Math. J.} 53 (2004), no. 4, 1037--1054. 
	
	\bibitem{GN}{\sc Gkikas, K. T; Nguyen, P.-T.}, On the existence of weak solutions of semilinear elliptic equations and systems with Hardy potentials. \textit{To appear in J. Differential Equations}, https://doi.org/10.1016/j.jde.2018.07.060.
	
	\bibitem{HMMY}{\sc Garc\'{i}a-Huidobro, M.; Man\'{a}sevich, R.; Mitidieri, E.; Yarur, C. S.}, Existence and nonexistence of positive singular solutions for a class of semilinear elliptic systems. \textit{Arch. Rational Mech. Anal.} 140 (1997), no. 3, 253--284. 
	
	
	\bibitem{LL}{\sc  Lam, N.; Lu, G.} Elliptic equations and systems with subcritical and critical exponential growth without the Ambrosetti-Rabinowitz condition. \textit{J. Geom. Anal.} 24 (2014), no. 1, 118--143.
	
	
	\bibitem{FQT}{\sc Felmer, P; Quaas, A.; Tan, J.}, Positive solutions of nonlinear Schrodinger equation with the fractional Laplacian, \textit{Proc. R. Soc. Edinb., Sect. A}, Math. 142 (2012) 1--26.
	
	%\bibitem{MV}{\sc Marcus, M.; V\'{e}ron, L.}, Nonlinear second order elliptic equations involving measures, \textit{De Gruyter Series in Nonlinear Analysis and Applications} 21, De Gruyter, Berlin, 2014.
	
	\bibitem{Mit-1}{\sc  Mitidieri, E.}, Nonexistence of positive solutions of semilinear elliptic systems in $\Rn$,
	\textit{Differential Integral Equations} 9 (1996), 465--479. 
	
	\bibitem{NS}{\sc  Naito, Y; Sato, T.} Positive solutions for semilinear elliptic equations with singular forcing terms. \textit{J. Differential Equations} 235 (2007), no. 2, 439--483.
	
	\bibitem{TV}{\sc Nguyen, P. T; Veron, L}, Boundary singularities of solutions to semilinear fractional equations. Adv. Nonlinear Stud. 18 (2018), 237-267.
	
	\bibitem{PQS}{\sc Pol\'{a}\^{c}ik, P.; Quittner, P.; Souplet, P.} Singularity and decay estimates in superlinear problems via Liouville-type theorems. I. Elliptic equations and systems. \textit{Duke Math.} J. 139 (2007), no. 3, 555--579.
	
	\bibitem{QX1}{\sc  Quaas, A., Xia, A.}; Liouville type theorems for nonlinear elliptic equations and systems involving fractional Laplacian in the half space. \textit{Calc. Var. Partial Differential Equations} 52 (2015), no. 3-4, 641--659.
	
	\bibitem{QX2}{\sc  Quaas, A., Xia, A.}; A Liouville type theorem for Lane-Emden systems involving the fractional Laplacian. Nonlinearity 29 (2016), no. 8, 2279--2297. 
	
	\bibitem{QX3}{\sc  Quaas, A., Xia, A.}; Existence results of positive solutions for nonlinear cooperative elliptic systems involving fractional Laplacian. \textit{Commun. Contemp. Math.} 20 (2018), no. 3, 1750032, 22 pp
	
	\bibitem{R}{\sc Rabinowitz, P. H.}; Minimax methods in critical point theory with applications to differential equations. \textit{CBMS Regional Conference Series in Mathematics}, 65. Published for the Conference Board of the Mathematical Sciences, Washington, DC; by the American Mathematical Society, Providence, RI, 1986. viii+100 pp. ISBN: 0-8218-0715-3. 
	
	\bibitem{RS1}{\sc  Ros-Oton, X.; Serra, J.}, The Dirichlet problem for the fractional Laplacian: regularity up to the boundary. \textit{J. Math. Pures Appl.} (9) 101 (2014), no. 3, 275--302. 
		
	\bibitem{RS2} {\sc  Ros-Oton, X.; Serra, J.},  The extremal solution for the fractional Laplacian. \textit{Calc. Var. Partial Differential Equations} 50 (2014), no. 3-4, 723--750. 
	
	\bibitem{RS3}{\sc Ros-Oton, X.; Serra, J.}, The Pohozaev identity for the fractional Laplacian. \textit{Arch. Ration. Mech. Anal.} 213 (2014), no. 2, 587--628.
	
	\bibitem{RZ}{\sc Reichel, W.; Zou, H.}, Non-existence results for semilinear cooperative elliptic
	systems via moving spheres, \textit{J. Differential Equations} 161 (2000), 219--243.
	
	\bibitem{SZ}{\sc Serrin, J.; Zou, H.}, Existence of positive solutions of the Lane-Emden system, \textit{Atti Sem. Mat. Fis. Univ. Modena} 46 (1998), suppl., 369-- 380. 
	
	\bibitem{Si} {\sc Silvestre, L.}, Regularity of the obstacle problem for a fractional power of the Laplace operator. \textit{Comm. Pure Appl. Math.} 60 (2007), no. 1, 67--112.
	
	\bibitem{Stein}{\sc Stein, E.M. }; Singular Integrals and Differentiability Properties of Functions, Princeton University Press, 1970.
\end{thebibliography}
\end{document}